\documentclass[a4paper, reqno, 10pt]{amsart}
\usepackage{amssymb}
\usepackage{extarrows}
\usepackage{bm}
\usepackage[centering]{geometry}
\usepackage{enumitem}
\usepackage{longtable}
\usepackage{multirow}
\usepackage{array}
\usepackage{etex}
\usepackage{pictex}
\usepackage{tikz}
\usepackage{graphics}

\usepackage[all]{xy}

\numberwithin{equation}{section}
\theoremstyle{definition}
\newtheorem{defn}{Definition}[section]

\newtheorem{exms}[defn]{Examples}

\theoremstyle{plain}
\newtheorem{cor}[defn]{Corollary}
\newtheorem{thm}[defn]{Theorem}
\newtheorem{lem}[defn]{Lemma}

\newtheorem{prop}[defn]{Proposition}

\newtheorem{defnfact}[defn]{Definition-Fact}

\newtheorem*{Freyd-Mitchell}{Freyd-Mitchell embedding Theorem}
\def\Fib{\operatorname{Fib}}
\def\CoFib{\operatorname{CoFib}}
\def\Weq{\operatorname{Weq}}
\def\Coker{\operatorname{Coker}}
\def\Hom{\operatorname{Hom}}
\def\Ext{\operatorname{Ext}}
\def\Ker{\operatorname{Ker}}

\def\TFib{\operatorname{TFib}}
\def\TCoFib{\operatorname{TCoFib}}
\def\Ho{\operatorname{Ho}}
\def\X{\mathcal{X}}
\def\Y{\mathcal{Y}}
\def\A{\mathcal{A}}

\title[ Homotopy categories and fibrant model structures]{Homotopy categories and fibrant model structures}
\author[Xue-Song Lu, Pu Zhang] {Xue-Song Lu, Pu Zhang$^*$ \\ \\  School of Mathematical Sciences \\
Shanghai Jiao Tong University,  \ Shanghai 200240, \ China }
\thanks{$^*$ Corresponding author}
\thanks{leocedar@sjtu.edu.cn \ \ \ \ pzhang$\symbol{64}$sjtu.edu.cn}
\thanks{Supported by National Natural Science Foundation of China, Grant No. 12131015 and Natural Science Foundation of Shanghai under Grant No. 23ZR1435100.}

\begin{document}
\maketitle
\begin{abstract} The homotopy category of a model structure on a weakly idempotent complete additive category
is proved to be equivalent to the additive quotient of the category of cofibrant-fibrant objects with respect to
the subcategory of cofibrant-fibrant-trivial objects.
A model structure on pointed category is fibrant,
if every object is a fibrant object. Fibrant model structures is explicitly described by trivial cofibrations, and also by fibrations.
Fibrantly weak factorization systems are introduced, fibrant model structures are constructed via fibrantly weak factorization systems, and
a one-one correspondence between fibrantly weak factorization systems and fibrant model structures is given.
Applications are given to rediscover
the $\omega$-model structures and the $\mathcal W$-model structures, and their relations with exact model structures are discussed.
\end{abstract}

\vskip10pt

\section{\bf Introduction}

This paper is three fold.
First, Qullen's homotopy category of a model structure on a weakly idempotent complete additive category
is expressed as the additive quotient of the category of cofibrant-fibrant objects with respect to
the subcategory of cofibrant-fibrant-trivial objects.

\vskip5pt

Second, a fibrant model structure
on a weakly idempotent complete additive category
is explicitly described by trivial cofibrations, and  also by fibrations.
To construct fibrant model structures,
fibrantly weak factorization systems are introduced;
and a one-one correspondence between  fibrant model structures and fibrantly weak factorization systems is given.
We will see that, even on exact categories,
fibrant model structures are in general not exact model structures.

\vskip5pt

Finally, as applications we rediscover two concrete fibrant model structures,
namely, the $\omega$-model structures on weakly idempotent complete exact categories,
and the $\mathcal W$-model structures on weakly idempotent complete additive categories.
Their relations with exact model structures are discussed and various examples are illustrated.

\subsection{The homotopy category of a model structure}

\vskip5pt

The homotopy category is a main object of study.
If $\A$ is a pointed model category, by introducing the  loop (suspension) functor and (co)fibration sequences,
D. Quillen has found that the homotopy category $\Ho(\mathcal A)$ satisfies  the similar axioms of a triangulated category ([Q1, p. 3.10]).
In fact, even if $\A$ is an additive model category, $\Ho(\mathcal A)$ is not triangulated, in general,
but pretriangulated ([H1, 6.5], [BR, II, 1.1]).
Anyway, the most known triangulated categories arise as $\Ho(\mathcal A)$.

\vskip5pt

The following result will be proved in Section 3.

\begin{thm}\label{hoA} \ Let $(\CoFib, \Fib, \Weq)$ be a model structure on a weakly idempotent complete additive category $\mathcal A$.
Then the homotopy category ${\rm Ho}(\mathcal A)$ is equivalent as an additive category to the quotient $(\mathcal C\cap \mathcal F)/(\mathcal C\cap \mathcal F\cap \mathcal W)$, where $\mathcal C,  \mathcal F, \mathcal W$ are respectively the class of cofibrant objects,
the class of fibrant objects, and the class of trivial objects.
\end{thm}

This result is known for exact model structures,
see J. Gillespie \cite[Proposition 4.4]{G2};
and for the $\omega$-model structures on weakly idempotent complete exact categories,
see A. Beligiannis and I. Reiten \cite[VIII, Theorem 4.2]{BR}  and \cite[Theorem 1.1]{CLZ}.
Moreover, if the exact model structure is given by a hereditary Hovey triple on a weakly idempotent complete exact category,
then $\mathcal C\cap \mathcal F$ is a Frobenius category with $\mathcal C\cap \mathcal F\cap \mathcal W$ as the class of
the projective-injective objects, see J. \v{S}\'{t}ov\'{i}\v{c}ek \cite[Theorem 6.21]{S} (also H. Becker  \cite[Proposition 1.1.14]{Bec}).

\vskip5pt

Quillen has introduced  (left, right) homotopy relation on pointed model category $\mathcal M$
(thus,  $\mathcal M$ has limits and colimits).
This gives an equivalent relation on $\mathcal C\cap \mathcal F$,
and the corresponding quotient category is just ${\rm Ho}(\mathcal{M})$ ([Q1, p.1.13]), up to an equivalence of category.
However, different from the proofs of the known cases mentioned above, here we can not use the homotopy relation, since
the existence of push-outs and pull-backs is not guaranteed  in a  weakly idempotent complete additive category.
Instead, one of the key steps is to use the fact that a kind of localization categories can be described as
additive quotient categories. See Theorem \ref{HoCcapF} and Lemma \ref{locilzationasquotient}.

\subsection{Descriptions of fibrant model structures}

Abelian model structures on abelian categories have been introduced by
M. Hovey, and this has been extended to exact model structures on exact categories by J. Gillespie.
The Hovey correspondence gives a way of constructing  exact model structures via
two complete cotorsion pairs (\cite[Theorem 2.2]{H2}; [G3, 3.3]; \cite[6.9]{S}).
Many model structures known in algebra are exact (abelian).
A. Beligiannis and I. Reiten construct a $\omega$-model structure on a weakly idempotent complete exact category
(\cite[Theorem 4.6]{BR}, \cite[Theorem 1.3]{CLZ}), which uses one hereditary complete cotorsion pair, but is not exact in general.

\vskip5pt

A category is {\it pointed} ([Q1]) if it has a zero object.
An additive category is {\it weakly idempotent complete},
provided that every splitting monomorphism has a cokernel,
or equivalently,
every splitting epimorphism has a kernel (see e.g. \cite[7.1, 7.2]{Bu}).

\vskip5pt

Let $(\CoFib, \Fib, \Weq)$ be a model structure on pointed category $\mathcal A$. We use the following notations. Put
$\TCoFib = \CoFib\cap\Weq$ and $\TFib = \Fib\cap\Weq$. Denote by $\mathcal C$ the class of cofibrant objects $X$, i.e.,
$0 \longrightarrow X$ is in $\CoFib$, $\mathcal F$ the class of fibrant objects $X$, i.e.,
$X \longrightarrow 0$ is in $\Fib$, and $\mathcal W$ the class of trivial objects $X$, i.e.,
$0 \longrightarrow X$ is in $\Weq$, or equivalently, $X \longrightarrow 0$ is in $\Weq$.
Put
${\rm T}\mathcal C = \mathcal C\cap\mathcal W$ and ${\rm T}\mathcal F  = \mathcal F\cap\mathcal W$.
We will also use these notations
$$\CoFib, \ \ \Fib, \ \ \Weq, \ \ \TCoFib, \ \ \TFib, \ \ \mathcal C, \ \ \mathcal F, \ \ \mathcal W, \ \ {\rm T}\mathcal C, \ \ {\rm T}\mathcal F$$
even if $(\CoFib, \Fib, \Weq)$ has not yet been known a model structure, but it will be a candidate for  a model structure.

\vskip5pt

A model structure on a pointed category will be said to be {\it fibrant},
if every object is a fibrant object. Fibrant model structures
on weakly idempotent complete additive categories, can be explicitly described by the trivial cofibrations, and also by the fibrations,
as the following result shows.

\begin{thm}\label{fibrant} \ Let  $(\CoFib, \Fib, \Weq)$ be a model structure on weakly idempotent complete additive category $\mathcal A$.
Then the followings are equivalent.

  \vskip5pt

  $(1)$ \ It is a fibrant model structure.

  \vskip5pt

  $(2)$ \ $\TCoFib=\{\text{splitting monomorphism} \ f \ |\ \Coker f\in {\rm T}\mathcal C \}$.

\vskip5pt

  $(3)$ \   $\Fib = \{ \text{morphism} \ f \ | \ f \ \mbox{is} \ \Hom_{\mathcal A}({\rm T}\mathcal C, -)\mbox{-epic}\}$.

\vskip5pt

If this is the case, then the homotopy category $\Ho(\mathcal A)\cong \mathcal C/{\rm T}\mathcal C$, as additive categories. \end{thm}

\subsection{Fibrantly weak factorization systems and construction of fibrant model structures}

Weak factorization systems have been studied in \cite{AHRT} and \cite{S}. See Definition \ref{WFS}.
To construct fibrant model structures, we introduce {\it fibrantly weak factorization systems.}

\begin{defn}\label{FWFS}  \ Let $\mathcal{A}$ be a pointed  category.
A pair $({\rm L},{\rm R})$ of classes of morphisms is a fibrantly weak factorization system on $\mathcal A$, provided that the following conditions are satisfied:

\vskip5pt

$(1)$ \ $({\rm L}, {\rm R})$ is a weak factorization system $($cf. Definition \ref{WFS}$)$.

\vskip5pt

$(2)$ \ For $f: A\longrightarrow B$ and $g:B\longrightarrow C$, if $f, gf\in {\rm R}$, then $g\in {\rm R}$.

\vskip5pt

$(3)$ \ The class $\mathcal L\cap \mathcal R$ of objects  is contravariantly finite in $\mathcal A$,
where $$\mathcal L =\{X \ | \ 0\longrightarrow X \ \mbox{is in} \ {\rm L}\}, \ \ \ \ \mathcal R =\{X \ | \ X\longrightarrow 0 \ \mbox{is in} \ {\rm R}\}.$$
\end{defn}

\vskip5pt

Assume that $(\CoFib, \TFib)$ is a fibrantly weak factorization system on $\mathcal A$. We stress that, although here we use the notations
$\CoFib$ and $\TFib$, however, at this moment they are not assumed to be from a model structure. Put

\begin{equation}\begin{aligned}
\mathcal C & = \{ \mbox{object} \ X \ | \ \mbox{morphism} \ 0\longrightarrow X \ \mbox{is in} \ \CoFib\} \\
\mathcal W & = \{ \mbox{object} \ X \ | \ \mbox{morphism}  \ X\longrightarrow 0 \ \mbox{is in} \ \TFib\} \\
{\rm T}\mathcal C & = \mathcal C\cap \mathcal W\\
\Fib & = \{ \text{morphism} \ f \ | \ f \ \mbox{is} \ \Hom_{\mathcal A}({\rm T}\mathcal C, -)\mbox{-epic}\}
\\
\TCoFib  & =\{\text{splitting monomorphism} \ f \ |\ \Coker f\in {\rm T}\mathcal C \}\\
\Weq & = \TFib\circ\TCoFib.\end{aligned}\end{equation}

\vskip5pt

For a class $\mathcal U$ of objects of category $\mathcal A$,
a morphism $f$ of  $\mathcal A$ is said to be {\it $\Hom_{\mathcal{A}}(\mathcal{U},-)$-epic},
provided that $\Hom_{\mathcal{A}}(U, f)$ is epic for any $U\in \mathcal U$.

\vskip5pt

\begin{thm}\label{mainthm} \ Let $\mathcal{A}$ be a weakly idempotent complete additive category.
If  $(\CoFib, \Fib, \Weq)$ is a fibrant model structure on $\mathcal A$, then
$(\CoFib, \TFib)$ is a fibrantly weak factorization system on $\mathcal A$, and
$\Fib$ and $\TCoFib$ are given as in $(1.1)$.

\vskip5pt

Conversely, if $(\CoFib, \TFib)$ is a fibrantly weak factorization system on $\mathcal A$, then $(\CoFib, \Fib, \Weq)$ is a fibrant model structure on $\mathcal A$,
where $\Fib$ and $\Weq$ are given as in $(1.1)$.
\end{thm}

The second part of Theorem \ref{mainthm} can be regarded as a constructive result,
in sense that any fibrant model structure on a weakly idempotent complete additive category, can be constructed
from a fibrantly weak factorization system, where $\Fib$ and $\TCoFib$ are explicitly given in $(1.1)$. 
This is sharpen by the following one-one correspondence
between fibrantly weak factorization systems and fibrant model structures.

\begin{thm} \label{correspondence} \ Let $\A$ be a weakly idempotent complete additive category, 
$\Omega$ the class of fibrantly weak factorization systems on $\A$, and $\Gamma$ the class of fibrant model structures on $\A$. Then the map
$$\Phi: \Omega \longrightarrow \Gamma \ \ \ \mbox{given by} \ \ \ (\CoFib, \TFib)\mapsto (\CoFib, \Fib, \Weq)$$
where $\Fib$ and $\Weq$ are given as in $(1.1)$,  and the map $$\Psi: \Gamma\longrightarrow \Omega  \ \ \ \mbox{given by} \ \ \
(\CoFib, \Fib, \Weq)\mapsto (\CoFib, \TFib)$$
give a one-one correspondence between $\Omega$ and $\Gamma$.
\end{thm}

This construction of fibrant model structures is on weakly idempotent complete additive categories. Even on exact categories, 
they are not exact model structures, 
and $(\mathcal C, {\rm T}\mathcal F)$ and $({\rm T}\mathcal C, \mathcal F)$ are not cotorsion pairs, in general. See Example \ref{examp}.

\subsection{The organization} In Section 2 we recall basics on localization of categories,  model structures and weak factorization systems, which will be used throughout the paper.

\vskip5pt

The homotopy category of a model structure on a weakly idempotent complete additive category is studied, and Theorems \ref{HoCcapF} and \ref{hoA} are proved, in Section 3.

\vskip5pt

Section 4 is devoted to the properties of fibrantly weak factorization systems and the proofs of Theorems \ref{fibrant}, \ref{mainthm} and \ref{correspondence}.

\vskip5pt

In Section 5, applying Theorem \ref{mainthm} to two special fibrant model structures, we rediscover $\omega$-model structures and $\mathcal W$-model structures,
respectively given by A. Beligiannis and I. Reiten [BR], and by A. Beligiannis [Bel].
Relationships among the both kinds of model structures and exact (respectively, abelian) model structures are discussed,
and various examples are listed.

\vskip5pt

Finally, the dual version of Section 4 are stated in Section 6.

\section{\bf Preliminaries}

\subsection{Splitting short exact sequences in additive categories} We need often the following well-known fact.

\begin{defnfact} \label{defnfact} \  A morphism sequence
$0\longrightarrow A\stackrel f \longrightarrow B\stackrel g \longrightarrow C\longrightarrow 0$ in an additive category $\mathcal A$ is
a splitting short exact sequence, provided that the following equivalent conditions are satisfied:

 \vskip5pt

  $(1)$ \ $f$ is a splitting monomorphism, and $g$ is the cokernel of $f$.

  \vskip5pt

  $(2)$ \ $g$ is a splitting epimorphism, and $f$ is the kernel of $g$.

  \vskip5pt

  $(3)$ \ There are morphisms  $f': B\longrightarrow A$ and $g':C\longrightarrow B$ such that $$f'f={\rm Id}_A,  \ \ gg'= {\rm Id}_C, \ \ ff'+g'g={\rm Id}_B.$$

  \vskip5pt

  $(4)$ \ There is a commutative diagram
  \[
  \xymatrix{
  0\ar[r]& A\ar@{=}[d]\ar[r]^-{f}&B\ar[d]^-{{\left(\begin{smallmatrix} f'\\g\end{smallmatrix}\right)}}\ar[r]^-{g}\ar[r]&C\ar@{=}[d]\ar[r]&0\\
  0\ar[r]& A\ar[r]^-{\left(\begin{smallmatrix} 1\\0\end{smallmatrix}\right)}&A\oplus C\ar[r]^-{(0, 1)}\ar[r]&C\ar[r]&0
  }
  \]
  where  $\left(\begin{smallmatrix} f'\\g\end{smallmatrix}\right)$ is an  isomorphism.

\vskip5pt

If this is the case, then any morphism $f': B\longrightarrow A$ with $f'f={\rm Id}_A$ has the kernel and $\Ker f' = \Coker f;$
and any morphism $g': C\longrightarrow B$ with $gg'= {\rm Id}_C$ has the cokernel and $\Coker g' = \Ker g$.
\end{defnfact}

\subsection{Additive quotient categories}
Let $\mathcal A$ be an additive category, and $\mathcal U$ a full additive subcategory (i.e., a full subcategory closed under finite coproducts).
Quotient category $\mathcal A/\mathcal U$ has the same objects as $\mathcal A$, and
\[
\Hom_{\mathcal A/\mathcal U}(X,Y)=\Hom_{\mathcal A}(X,Y)/\Hom_{\mathcal A}(X, \mathcal U, Y)
\]
where $\Hom_{\mathcal A}(X, \mathcal U, Y)$ is the subgroup $\{f\in \Hom_{\mathcal A}(X,Y) \ | \ f \ \text{factors \ through\ object \ in \ } \mathcal U\}$.
Then
$\mathcal A/\mathcal U$ is an additive category with the canonical additive functor $P: \mathcal A\longrightarrow \mathcal A/\mathcal U$,
which sends objects in $\mathcal U$ to the zero object, and has the universal property with this respect. That is, If
$Q: \mathcal A\longrightarrow \mathcal B$ is an additive functor also sending objects in $\mathcal U$ to the zero object, then
there exists a unique additive functor $\bar Q: \mathcal A/\mathcal U \longrightarrow \mathcal B$ such that $Q = \bar Q\circ P$.
This amounts that $P$ sends morphisms in $\Hom_{\mathcal A}(X, \mathcal U, Y)$ to zero morphisms, and has the universal property with this respect.
If $\mathcal U$ is closed under direct summands, then object $X \cong 0$ in $\mathcal A/\mathcal U$ if and only if $X\in \mathcal U$.

\vskip5pt

\subsection{Localization of categories}  Let $\mathcal S$ be a class of morphisms of a category $\mathcal A$.
By definition the localization category ${\mathcal A}[\mathcal S^{-1}]$ of $\mathcal A$ with respect to $\mathcal S$
is a category together with localization functor $L:\mathcal A\longrightarrow {\mathcal A}[\mathcal S^{-1}]$,
which sends morphisms in $\mathcal S$ to isomorphisms, and has the universal property with this respect.

\vskip5pt

Recall the construction of ${\mathcal A}[\mathcal S^{-1}]$,
via the path category ${\rm Path}Q(\mathcal A, \mathcal S)$ of quiver $Q(\mathcal A, \mathcal S)$,
given by P. Gabriel and M. Zisman \cite{GZ}. See also H. Krause \cite{K}.
Vertices  of $Q(\mathcal A, \mathcal S)$ are the objects of $\mathcal A$.
For two vertices $X$ and $Y$,
the set of arrows of $Q(\mathcal A, \mathcal S)$ from $X$ to $Y$ is the disjoint union
$$\Hom_\mathcal A(X, Y) \ \overset\bullet \bigcup \ \mathcal S(Y, X)$$
where $\mathcal S(Y, X)$ is the set of morphisms
$g: Y\longrightarrow X$ in $\mathcal S$. Thus, if $g: Y\longrightarrow X$ is a morphism $\mathcal S$, then in  $Q(\mathcal A, \mathcal S)$ one has
not only an arrow $g: Y\longrightarrow X$, but also a new added arrow $X\longrightarrow Y$, denoted by $g^{-1}: X\longrightarrow Y$. Also, for each vertex $X$ there is a loop
${\rm Id}_X: X\longrightarrow X$ in $Q(\mathcal A, \mathcal S)$, corresponding to the identity morphism of $X$.
By definition, objects of path category ${\rm Path}Q(\mathcal A, \mathcal S)$ are vertices of $Q(\mathcal A, \mathcal S)$
(thus, objects of $\mathcal A$); and morphisms of ${\rm Path}Q(\mathcal A, \mathcal S)$ from $X$ to $Y$ are paths from $X$ to $Y$ in $Q(\mathcal A, \mathcal S)$.
Then ${\mathcal A}[\mathcal S^{-1}]$ is
exactly the quotient category of ${\rm Path}Q(\mathcal A, \mathcal S)$, with respect to
the equivalent relations generated by the obvious relations:

\vskip5pt

(1) \ For each object $X$ of $\mathcal A$,  one has ${\rm Id}_X\sim e_X,$ where ${\rm Id}_X$ is the morphism
of ${\rm Path}Q(\mathcal A, \mathcal S)$ given by the loop at $X$, and $e_X$ is the morphism given by the path of length $0$ at $X$;

\vskip5pt

(2) \ For composable morphisms $s: X\longrightarrow Y$ and $t: Y\longrightarrow Z$ of $\mathcal A$,  one has
$t\circ s \sim ts,$ where $t\circ s$ is the morphism  in ${\rm Path}Q(\mathcal A, \mathcal S)$ given by the concatenation of arrows $s$ and $t$, and $ts$ is the morphism given by arrow $ts$.

\vskip5pt

(3) \ For any morphism \ $g: Y\longrightarrow X$ in $\mathcal S$, one has $g\circ g^{-1} \sim {\rm Id}_X$ and $g^{-1}\circ g \sim {\rm Id}_Y$.

\vskip5pt Thus, if $s\in S$ and $t\in {\rm Mor}(\mathcal A)$ such that $ts\in S$ (respectively, $st\in \mathcal S$),
then $(ts)^{-1}\circ t=s^{-1}$ (respectively, $t \circ (st)^{-1}=s^{-1}$) in ${\mathcal A}[S^{-1}]$.

\vskip5pt

In conclusion, localization category ${\mathcal A}[\mathcal S^{-1}]$ has the same objects as $\mathcal A$,
and a morphism $\delta: X\longrightarrow Y$ in ${\mathcal A}[\mathcal S^{-1}]$ is an equivalence class. If $\mathcal S$ is closed under compositions,
then $\delta$ can be written as $\delta={g_n}^{-1}f_n\cdots {g_1}^{-1}f_1$, where
\[
\xymatrix{
X\ar[r]^-{f_1}& A_1& B_1\ar[l]_-{g_1}\ar[r]^-{f_2} & A_2& B_2\ar[l]_-{g_2}\ar[r]^-{f_3}& \ldots \ar[r]^-{f_{n}}& A_{n}& B_n=Y\ar[l]_-{g_{n}}
}
\]
is a sequence of morphisms of $\mathcal A$ and $g_1,\ldots, g_n\in \mathcal S$. The localization functor $L: \mathcal A \longrightarrow {\mathcal A}[\mathcal S^{-1}]$
is the identity on objects, and sends morphism $f$ of $\mathcal A$ to the equivalence class where $f$ lies in.

\begin{lem} \label{locilzationasquotient} \ {\rm ([K, Lemmas 2.2.1,  2.2.2])} \ Let $\mathcal A$ be an additive category.

\vskip5pt

$(1)$ \  Let $\mathcal S$ be a class of morphisms of $\mathcal A$ such that ${\rm Id}_X\in\mathcal S$ for every object $X\in \mathcal A$,
and that $\mathcal S$ is closed under finite coproducts.
Then ${\mathcal A}[\mathcal S^{-1}]$ is an additive category,
and $L: \mathcal A \longrightarrow {\mathcal A}[\mathcal S^{-1}]$ is an additive functor.

\vskip5pt

$(2)$ \  Let $\mathcal U$ be a full additive subcategory of $\mathcal A$. Set
$$\mathcal S = \{f\in {\rm Mor}\mathcal A \ | \ \bar f \ \mbox{is an isomorphism in} \ \mathcal A/\mathcal U\}.$$
Then the canonical functor $\mathcal A\longrightarrow \mathcal A/\mathcal U$ induces an isomorphism
${\mathcal A}[\mathcal S^{-1}]\cong \mathcal A/\mathcal U$ of  additive categories.

\end{lem}

\subsection{Model structures} Morphism $g: A'\longrightarrow B'$ is a retract of morphism $f: A\longrightarrow B$, provided that
there is a commutative diagram
\[
\xymatrix@R=0.5cm{
    A'\ar[r]^{\varphi_1}\ar[d]_g & A\ar[r]^{\psi_1}\ar[d]_{f} & A'\ar[d]_{g} \\
    B'\ar[r]^{\varphi_2} & B\ar[r]^{\psi_2} & B'
}
\]
such that $\psi_1 \varphi_1 ={\rm Id}_{A'}$ and $\psi_2 \varphi_2 = {\rm Id}_{B'}$.

\vskip5pt

For $f:A\longrightarrow B$ and $g:C\longrightarrow D$, one writes $f\ \square\ g$ provided that for any commutative diagram
\begin{equation}
\xymatrix@R=0.5cm{
A\ar[r]^-{u}\ar[d]_-{f} & C \ar[d]^-{g}\\
B\ar[r]^-v & D
}
\end{equation}
there is a lifting $h:B\longrightarrow C$ such that $hf=u$ and $gh=v$. In this case, $f$ is said to have {\it the left lifting property} with respect to $g$,
and  $g$ is said to have {\it the right lifting property} with respect to $f$.
For a class $\mathcal S$ of morphisms, denoted by ${\rm LLP}(\mathcal S)$
(${\rm RLP}(\mathcal S)$, respectively) the class of morphisms $f$ such that $f\ \square\ s$ ($s\ \square\ f$, respectively) for all $s\in \mathcal S$.

\begin{defn} \label{ms} \  {\rm ([Q1], [Q2])} \ A model structure on a category $\mathcal M$ is a triple
$(\CoFib$, $\Fib$, $\Weq)$ of classes of morphisms,
where the morphisms in the three classes are respectively called {\it cofibrations, fibrations, and weak equivalences},
satisfying the following axioms:

\vskip5pt

{\bf Two out of three axiom:} \ Let $X\xlongrightarrow{f}Y\xlongrightarrow{g}Z$ be morphisms in $\mathcal M$. If two of the morphisms $f, \ g, \ gf$ are weak equivalences,
then so is the third one.

\vskip5pt

{\bf Retract axiom:}  \  $\CoFib$, $\Fib$, $\Weq$ are closed under retract.

\vskip5pt

{\bf Lifting axiom:} \  Let $f\in \CoFib$ and $g\in \Fib$. If either $f\in \Weq $ or $g\in \Weq$, then $f\ \square\ g$.

\vskip5pt

{\bf Factorization axiom:}  \ Any morphism $f: X\longrightarrow Y$ admits factorizations \ $f=pi$ and  \ $f=qj$, where \ $i\in \CoFib\cap \Weq$, \ $p\in \Fib$,  \ $j\in \CoFib$, and \ $q\in \Fib\cap \Weq$.
\end{defn}

The morphisms in $\TCoFib: =\CoFib\cap \Weq$ and $\TFib: = \Fib\cap \Weq$ are called {\it trivial cofibrations} and {\it trivial fibrations}, respectively.

\vskip5pt

For a model structure $(\CoFib$, $\Fib$, $\Weq)$ on a pointed category $\mathcal M$,
an object $X$ is {\it cofibrant} if $0\longrightarrow X$ is a cofibration;  it is {\it fibrant} if $X\longrightarrow 0$ is a fibration;
and it is {\it trivial} if $0 \longrightarrow X $ is a weak equivalence, or, equivalently,
$X\longrightarrow 0$ is a weak equivalence.
An object is {\it trivially cofibrant} (respectively, {\it trivially fibrant}) if it is both trivial and cofibrant (respectively, both trivial and fibrant).
Denote the class of cofibrant objects, fibrant objects, trivial objects, trivially cofibrant objects and trivially fibrant objects,
by $\mathcal C$, \ $\mathcal F$, \ $\mathcal W$, \ ${\rm T}{\mathcal C} = \mathcal C\cap \mathcal W$ and ${\rm T}{\mathcal F} = \mathcal F\cap \mathcal W$,  respectively.

\vskip5pt

We need the following fundamental results. For proofs one can find for examples in [Q1], [Q2], [H1], [BR].

\begin{lem}\label{morphism} \ Let $(\CoFib$, $\Fib$,  $\Weq)$ be a model structure on category $\mathcal M$. Then

\vskip5pt

$(1)$ \ \ Any two classes of $\CoFib$, \  $\Fib$, \  $\Weq$ determine the third uniquely. Namely,
$$\CoFib = {\rm LLP}(\TFib), \ \ \ \TCoFib = {\rm LLP}(\Fib)$$ and
$$\Fib = {\rm RLP}(\TCoFib), \ \ \ \TFib = {\rm RLP}(\CoFib), \ \ \ \Weq = \TFib\circ \TCoFib.$$

\vskip5pt

$(2)$ \ \ $\CoFib$, \ $\Fib$, \ $\TCoFib$ and \ $\TFib$ are closed under compositions.

\vskip5pt

$(3)$ \ Isomorphisms are  cofibrations, fibrations, and weak equivalences.

 \vskip5pt

$(4)$ \ $\CoFib$ and $\TCoFib$ are closed under push-outs, i.e., if square $(2.1)$ is a push-out square
and $f\in \CoFib$ $($respectively, $f\in \TCoFib$$)$, then $g\in \CoFib$ $($respectively, $g\in \TCoFib$$)$.

\vskip5pt

In particular, if $\mathcal M$ is pointed and $f\in \CoFib$ $($respectively, $f\in \TCoFib$$)$,
and if $f$ has cokernel, then $\Coker f\in \mathcal C$
$($respectively, $\Coker f\in {\rm T}\mathcal C$$)$.

\vskip5pt

$(5)$ \ $\Fib$ and $\TFib$ are closed under pull-backs, i.e., if square $(2.1)$ is a pull-back square and
$g\in \Fib$ $($respectively, $g\in \TFib$$)$, then $f\in \Fib$ $($respectively, $f\in \TFib$$)$.

\vskip5pt

In particular, if $\mathcal M$ is pointed and $f\in \Fib$ $($respectively, $f\in \TFib$$)$,
and if $f$ has kernel, then $\Ker f\in \mathcal F$
$($respectively, $\Ker f\in {\rm T}\mathcal F$$)$.

\vskip5pt

$(6)$ \ Assume that $\mathcal M$ is additive.
If \ $X\in\mathcal C$ $($respectively, $X\in{\rm T}\mathcal C)$, then any splitting monomorphism with cokernel $X$ belongs to $\CoFib$
$($respectively, $\TCoFib$$)$.

\vskip5pt

$(7)$ \ Assume that $\mathcal M$ is additive.
If \ $X\in\mathcal F$ $($respectively, $X\in{\rm T}\mathcal F)$, then any splitting epimorphism with kernel $X$ belongs to $\Fib$
$($respectively, $\TFib$$)$.

\vskip5pt

$(8)$ \ \ If $\mathcal M$ is additive, then $\CoFib$, \ $\Fib$, \ $\TCoFib$, \ $\TFib$ and \ $\Weq$ are closed under finite coproducts.\end{lem}

\begin{lem}\label{object} \ Let $(\CoFib$, $\Fib$,  $\Weq)$ be a model structure on pointed category $\mathcal M$. Then

\vskip5pt

$(1)$ \  \ $\mathcal C$ and ${\rm T}{\mathcal C}$ are contravariantly finite in $\mathcal M$.

  \vskip5pt

  $(2)$ \  \ $\mathcal F$ and ${\rm T}{\mathcal F}$  are covariantly finite in $\mathcal M$.

  \vskip5pt

  $(3)$ \ \ $\mathcal C$ and ${\rm T}{\mathcal C}$ are closed under finite coproducts, direct summands, and direct productands.
  $($Object $X$ is a productand, if there is an object $Y$ such that product $X\prod Y$ exists.$)$

 \vskip5pt

  $(4)$ \ \ $\mathcal F$ and ${\rm T}{\mathcal F}$ are closed under finite products, direct summands, and direct productands.

\vskip5pt

$(5)$ \ \ If $\mathcal M$ is additive, then $\mathcal W$ is closed under finite coproducts and direct summands.
\end{lem}

\begin{lem}\label{CCoFib}
Let $(\CoFib$, $\Fib$,  $\Weq)$ be a model structure on pointed category $\mathcal M$, and $f: A\longrightarrow B$ a morphism.

\vskip5pt

$(1)$ \ If $A\in \mathcal C$, $f\in \CoFib$, then $B\in \mathcal C$.

\vskip5pt

$(2)$ \ If $A\in {\rm T}\mathcal C$, $f\in \TCoFib$, then $B\in {\rm T}\mathcal C$.

\vskip5pt

$(3)$ \ If $B\in \mathcal F$, $f\in \Fib$, then $A\in \mathcal F$.

\vskip5pt

$(4)$ \ If $B\in {\rm T}\mathcal F$, $f\in \TFib$, then $A\in {\rm T}\mathcal F$.

\end{lem}

\begin{proof} \ All these assertions follow from Lemma \ref{morphism}(2). For example, for $(1)$, by definition $0\longrightarrow A$ is in $\CoFib$. If follows that the composition
$(0\longrightarrow B) =f\circ (0\longrightarrow A)$ is in $\CoFib$.
\end{proof}

For a class $\mathcal U$ of objects of category $\mathcal A$,
a morphism $f$ of  $\mathcal A$ is said to be {\it $\Hom_{\mathcal{A}}(-, \mathcal{U})$-epic},
provided that $\Hom_{\mathcal{A}}(f, U)$ is epic for any $U\in \mathcal U$.

\vskip5pt

\begin{lem}\label{split} \ {\rm (\cite[VIII, Lemma 1.1]{BR})} \ Let $(\CoFib, \Fib, \Weq)$ be a model structure on pointed category $\mathcal M$. Then

\vskip5pt

$(1)$ \ Any cofibration $f: A\longrightarrow B$ is $\Hom_{\mathcal M}(-, {\rm T}\mathcal F)$-epic$;$
in particular, if $A\in {\rm T}\mathcal F$ then $f$ is a splitting monomorphism.

\vskip5pt

$(2)$ \ Any trivial cofibration $f: A\longrightarrow B$ is $\Hom_{\mathcal M}(-, \mathcal F)$-epic$;$
in particular, if $A\in \mathcal F$ then $f$ is a splitting monomorphism.

\vskip5pt

$(3)$ \ Any fibration $f: A\longrightarrow B$ is $\Hom_{\mathcal M}({\rm T}\mathcal C,-)$-epic$;$
in particular, if $B\in {\rm T}\mathcal C$ then  $f$ is a splitting epimorphism.

\vskip5pt

$(4)$ \ Any trivial fibration $f: A\longrightarrow B$ is $\Hom_{\mathcal M}(\mathcal C, -)$-epic$;$
in particular, if $B\in \mathcal C$ then $f$ is a splitting epimorphism.\end{lem}





\subsection{Exact model structures and the Hovey correspondence}
A model structure on exact category $\A$ is {\it exact} ([G2]), if
cofibrations are precisely inflations with cofibrant cokernel, and fibrations are precisely
deflations with fibrant kernel. If this is the case, trivial cofibrations are precisely inflations with trivially cofibrant cokernel,
and trivial fibrations are precisely deflations with trivially fibrant kernel.
If $\mathcal A$ is abelian, then this goes to the notion of {\it abelian model structure} ([H2]).

\vskip5pt

\vskip5pt

{\it A Hovey triple} (\cite{H2}, \cite{G2}) in exact category $\mathcal A$ is a triple \ $(\mathcal C, \mathcal F, \mathcal W)$ \ of classes of objects such that
\  $\mathcal W$ is {\it thick} in $\mathcal A$ (i.e., $\mathcal W$ is closed under direct summands,
and if two out of the three terms in an admissible exact sequence  are in $\mathcal W$, then so is the third one); and that
both $(\mathcal C \cap \mathcal W, \ \mathcal F)$ and \ $(\mathcal C, \ \mathcal F \cap \mathcal W)$ \  are complete cotorsion pairs.

\begin{thm} \label{hoveycorrespondence} {\rm (The Hovey correspondence) \ ([H2, Theorem 2.2]; [G2, 3.3]; \cite[6.9]{S})} \ Let $\mathcal A$ be a weakly idempotent complete exact category.
Then there is a one-to-one correspondence between exact model structures and Hovey triples in $\mathcal A$, given by
 \ $({\rm CoFib}, \ {\rm Fib}, \ {\rm Weq})\mapsto (\mathcal{C}, \ \mathcal{F}, \ \mathcal W)$,
with the inverse \ \ $(\mathcal{C}, \ \mathcal{F}, \ \mathcal W) \mapsto ({\rm CoFib}, \ {\rm Fib}, \ {\rm Weq}).$
\end{thm}

A Hovey triple  $(\mathcal C, \mathcal F, \mathcal W)$ is {\it hereditary} if both the cotorsion pairs
$(\mathcal C \cap \mathcal W, \ \mathcal F)$ and \ $(\mathcal C, \ \mathcal F \cap \mathcal W)$ \  are hereditary.
An advantage of a hereditary Hovey triple is that the homotopy category of the corresponding exact model structure
is just the stable category of a Frobenius category, and hence triangulated.
J. Gillespie \cite[Theorem 1.1]{G3} has
constructed all the hereditary Hovey triples in an abelian category via two compatible hereditary complete cotorsion pairs.

\subsection{Weak factorization systems} \ Recall the basics on weak factorization systems from \cite{AHRT} and \cite{S}.

\vskip5pt

\begin{defn}\label{WFS} \ Let $\mathcal M$ be a category.
A pair $({\rm L},{\rm R})$ of classes of morphisms of $\mathcal M$ is a weak factorization system on $\mathcal M$,
provided that the following conditions are satisfied:

\vskip5pt

$(1)$ \ \ Both ${\rm L}$ and ${\rm R}$ are closed under retracts.

\vskip5pt

$(2)$ \ \ $l\ \square\ r$ for all $l\in {\rm L}$, $r\in {\rm R}$.

\vskip5pt

$(3)$ \ For any $f\in {\rm Mor}(\mathcal M)$, there is a factorization $f= r\circ l$ with $l\in {\rm L}$ and $r\in {\rm R}$.
\end{defn}

\vskip5pt

Typical examples of weak factorization systems come from model structures:
if $(\CoFib, \Fib, \Weq)$ is a model structure on category $\mathcal M$, then  both $(\CoFib, \TFib)$ and $(\TCoFib, \Fib)$ are weak factorization systems on $\mathcal M$.
For the proofs of the following facts we refer to \cite{AHRT} and \cite{S}. See also [Q1] and [Q2].

\begin{lem}\label{lemFWFS} \ Let $({\rm L},{\rm R})$ be a weak factorization system on  $\mathcal M$. Then

  \vskip5pt

  $(1)$ \ \ ${\rm L}={\rm LLP}({\rm R})$ and \  ${\rm R}={\rm RLP}({\rm L})$. In particular, isomorphisms  belong to ${\rm L}$ and ${\rm R}$.

    \vskip5pt

  $(2)$ \ \ ${\rm L}$ and ${\rm R}$ are closed under compositions. If $\mathcal M$ is additive, then ${\rm L}$ and ${\rm R}$  are closed under finite coproducts.

    \vskip5pt

  $(3)$ \ \ ${\rm L}$ is closed under push-outs.

\vskip5pt

In particular, if $\mathcal M$ is pointed and $f\in {\rm L}$,  and if $f$ has cokernel, then $0\longrightarrow \Coker f$ is in ${\rm L}$.

    \vskip5pt

$(4)$ \ ${\rm R}$ is closed under pull-backs.

\vskip5pt  In particular, if $\mathcal M$ is pointed and $f\in {\rm R}$, and if $f$ has kernel, then $\Ker f\longrightarrow 0$ is in ${\rm R}$.

  \vskip5pt

  $(5)$ \ Assume that $\mathcal M$ is additive.

\vskip5pt

If \ $0\longrightarrow X$ is in ${\rm L}$, then any splitting monomorphism with cokernel $X$ belongs to ${\rm L}.$

\vskip5pt

If \ $X\longrightarrow 0$ is in ${\rm R}$, then any splitting epimorphism with kernel $X$ belongs to $R$.

\end{lem}

\section{\bf The homotopy category}
By definition the homotopy category $\Ho(\mathcal M)$, of a model structure $(\CoFib, \Fib, \Weq)$ on category $\mathcal M$, is the localization category ${\mathcal M}[\Weq^{-1}]$.
If $\mathcal M$ is additive, then ${\mathcal M}[\Weq^{-1}]$ is additive and the localization functor $\mathcal M \longrightarrow \Ho(\mathcal M)$ is additive. See Lemma \ref{locilzationasquotient}.

\vskip5pt

\subsection{The homotopy categories of model structures} \ If $\mathcal M$ is a pointed category, then one has three full subcategories:
the subcategory $\mathcal C$ consisting of cofibrant objects,
the subcategory $\mathcal F$ consisting of fibrant object, and $\mathcal C\cap \mathcal F$. As $\Ho(\mathcal M)=\mathcal {M}[\Weq^{-1}]$,
we put $$\Ho(\mathcal C)=\mathcal {C}[(\Weq\cap {\rm Mor}(\mathcal C))^{-1}],  \ \ \ \ \Ho(\mathcal F)=\mathcal {F}[(\Weq\cap {\rm Mor}(\mathcal F))^{-1}]$$ and $$\Ho(\mathcal C\cap \mathcal F)=(\mathcal C\cap \mathcal F)[(\Weq\cap {\rm Mor}(\mathcal C\cap \mathcal F))^{-1}].$$ Denote the corresponding localization functor by $L$, $L_{\mathcal C}$, $L_{\mathcal F}$ and $L_{\mathcal C\cap \mathcal F}$, respectively. Then the canonical embedding $i: \mathcal C\cap \mathcal F\longrightarrow \mathcal M$ and the universal property of the localization functor induces
a unique functor $\overline{i}:\Ho(\mathcal C\cap \mathcal F)\longrightarrow\Ho(\mathcal M)$ such that the following diagram commutes:
\begin{equation}
\xymatrix{
\mathcal C\cap \mathcal F \ar[r]^-i\ar[d]_{L_{\mathcal C\cap \mathcal F}}& M\ar[d]^-L\\
\Ho(\mathcal C\cap \mathcal F)\ar@{..>}[r]^{\overline{i}}& \Ho(\mathcal M).
}
\end{equation}

\begin{thm}\label{HoCcapF} \ Let $\mathcal M$ be a pointed category with finite coproducts and products, and $(\CoFib, \Fib, \Weq)$ a model structure on $\mathcal M$. Then the canonical functor
  $\overline{i}:\Ho(\mathcal C\cap \mathcal F)\longrightarrow\Ho(\mathcal M)$ is an equivalence of categories.
\end{thm}

\begin{lem}\label{faithful}

 Let $\mathcal M$ be a pointed category with finite coproducts and products, and $(\CoFib, \Fib, \Weq)$ a model structure on $\mathcal M$. Then

 \vskip5pt

 $(1)$ \ Let $f, g: X\longrightarrow Y$ with $X\in \mathcal C$ and  $Y\in \mathcal C$. Suppose that there is a trivial fibration $h:Y\longrightarrow Z$ such that $hf=hg$.
 Then $L_{\mathcal C}(f)=L_{\mathcal C}(g)$ in $\Ho(\mathcal C)$.

 \vskip5pt

 $(2)$ \ Let $f, g: Y\longrightarrow Z$ with $Y\in \mathcal F$ and $Z\in \mathcal F$. Suppose that there is a trivial cofibration $h: X\longrightarrow Y$ such that $fh=gh$.
 Then $L_{\mathcal F}(f)=L_{\mathcal F}(g)$ in $\Ho(\mathcal F)$.

 \vskip5pt

 $(3)$ \ Let $f, g: X\longrightarrow Y$ with $X\in \mathcal C\cap \mathcal F$ and $Y\in \mathcal C\cap \mathcal F$.
 Suppose that there is a trivial fibration $h: Y\longrightarrow Z$ such that $hf=hg$. Then $L_{\mathcal C\cap\mathcal F}(f)=L_{\mathcal C\cap\mathcal F}(g)$ in $\Ho(\mathcal C\cap\mathcal F)$.

 \vskip5pt

 $(4)$ \ Let $f, g: Y\longrightarrow Z$ with $Y\in \mathcal C\cap\mathcal F$ and $Z\in \mathcal C\cap\mathcal F$.
 Suppose that there is a trivial cofibration  $h:X\longrightarrow Y$ such that $fh=gh$.
 Then $L_{\mathcal C\cap\mathcal F}(f)=L_{\mathcal C\cap\mathcal F}(g)$ in $\Ho(\mathcal C\cap\mathcal F)$.
\end{lem}
\begin{proof} \ We only prove $(1)$ and $(2)$. The assertion $(3)$ and $(4)$ can be similarly proved.

\vskip5pt

(1) \ Consider $(1,1): X\oplus X\longrightarrow X$ and its factorization  $(1, 1) = \sigma \circ (\partial_1, \partial_2)$,
where $(\partial_1,\partial_2): X\oplus X\longrightarrow \tilde{X}$ is in $\CoFib$ and $\sigma: \tilde{X}\longrightarrow X$ is in $\TFib$.
Since $X\in \mathcal C$, it follows from Lemma \ref{object}(3) that $X\oplus X\in \mathcal C$, and hence $\tilde{X}\in \mathcal C$, by Lemma \ref{CCoFib}(1).
The following commutative diagram
\[
\xymatrix{
X\oplus X\ar[r]^-{(f,g)}\ar[d]_-{(\partial_1,\partial_2)} & Y\ar[d]^-{h}\\
\tilde{X}\ar@{..>}[ur]^-{t}\ar[r]^-{fh\sigma}& Z
}
\]
gives a lifting $t$ such that $(f,g)= t(\partial_1,\partial_2)$. Since $\sigma\in \Weq\cap {\rm Mor}(\mathcal C)$, $L_{\mathcal C}(\sigma)$ is an isomorphism and hence
$L_{\mathcal C}(\partial_1)=L_{\mathcal C}(\sigma)^{-1}=L_{\mathcal C}(\partial_2)$. Thus $L_{\mathcal C}(f)=L_{\mathcal C}(t\partial_1)=L_{\mathcal C}(t)L_{\mathcal C}(\partial_1)=L_{\mathcal C}(t)L_{\mathcal C}(\partial_2)=L_{\mathcal C}(t\partial_2)=L_{\mathcal C}(g)$.

\vskip5pt

$(2)$ \ Consider the morphism $\left(\begin{smallmatrix} 1 \\ 1\end{smallmatrix}\right): Z\longrightarrow Z\prod Z$ and its factorization  $\left(\begin{smallmatrix} 1 \\ 1\end{smallmatrix}\right) =\left(\begin{smallmatrix} \partial_1 \\ \partial_2\end{smallmatrix}\right) \circ \sigma$,
where $\sigma: Z\longrightarrow \hat{Z}$ is in $\TCoFib$ and $\left(\begin{smallmatrix} \partial_1 \\ \partial_2\end{smallmatrix}\right): \hat{Z}\longrightarrow Z\prod Z$ is in $\Fib$.
Since $Z\in \mathcal F$, it follows from Lemma \ref{object}(4) that $Z\prod Z\in \mathcal F$, and hence $\hat{Z}\in \mathcal F$, by Lemma \ref{CCoFib}(3).
The following commutative diagram
\[
\xymatrix{&
X\ar[r]^-{\sigma fh}\ar[d]_-{h} & \hat{Z}\ar[d]^-{ \left(\begin{smallmatrix} \partial_1 \\ \partial_2\end{smallmatrix}\right)}\\
&Y\ar@{..>}[ur]^-{t}\ar[r]_-{\left(\begin{smallmatrix} f \\ g\end{smallmatrix}\right) }& Z\prod Z
}
\]
gives a lifting $t$ such that $\left(\begin{smallmatrix} f \\ g\end{smallmatrix}\right)= \left(\begin{smallmatrix} \partial_1 \\ \partial_2\end{smallmatrix}\right)t$.
Since $\sigma\in \Weq\cap {\rm Mor}(\mathcal F)$,  $L_{\mathcal F}(\partial_1)=L_{\mathcal F}(\sigma)^{-1}=L_{\mathcal F}(\partial_2)$. Thus $L_{\mathcal F}(f)=L_{\mathcal F}(\partial_1t)=L_{\mathcal F}(\partial_1)L_{\mathcal F}(t)=L_{\mathcal F}(\partial_2)L_{\mathcal F}(t)=L_{\mathcal F}(\partial_2t)=L_{\mathcal F}(g)$.
\end{proof}

\subsection{\bf Proof of Theorem \ref{HoCcapF}} \ We need to show that $\overline{i}$ is fully faithful and dense. For $X\in \mathcal M$, fix a factorization of morphism
$$(0\longrightarrow X) = ({\rm Q}X \stackrel {p_X} \longrightarrow X)\circ (0\longrightarrow {\rm Q}X)\eqno (3.2)$$
where ${\rm Q}X\in \mathcal C$ and  $p_X\in \TFib$,  such that if $X\in \mathcal C$ then ${\rm Q}X=X$ and $p_X={\rm Id}_X$.
Also, fix a factorization
$$(X\longrightarrow 0) = ({\rm R}X \longrightarrow 0)\circ (X\stackrel {i_X} \longrightarrow {\rm R}X)\eqno(3.3)$$
where ${\rm R}X\in \mathcal F$ and  $i_X\in \TCoFib$, such  that if  $X\in \mathcal F$ then ${\rm R}X=X$ and $i_X={\rm Id}_X$.
Since $L(i_X)$ and $L(p_{_{{\rm R}X}})$ are isomorphisms, one has $X\cong {\rm QR}X$ in $\Ho(M)$.
Since $p_{_{{\rm R}X}}: {\rm QR}X\longrightarrow {\rm R}X$ is a fibration with ${\rm R}X\in \mathcal F$,
by Lemma \ref{CCoFib}(3)  ${\rm QR}X\in \mathcal F$, and hence ${\rm QR}X\in \mathcal C\cap \mathcal F$. By $(3.1)$ one see that $\overline{i}$ is dense.

  \vskip5pt

To prove that $\overline{i}$ is faithful, define a functor $\overline{\rm Q}: \mathcal M\longrightarrow \Ho(\mathcal C)$ as follows. For $X\in \mathcal M$, let $\overline{\rm Q}X={\rm Q}X$.
For $f: X\longrightarrow Y$, choose ${\rm Q}f$ to be a lifting of
\[
\xymatrix{
0\ar[r]\ar[d]& {\rm Q}Y\ar[d]^-{p_Y}\\
{\rm Q}X\ar@{..>}[ur]^-{{\rm Q}f}\ar[r]^-{f\circ p_X}& Y
}
\]
such that if both $X$ and $Y$ are in $\mathcal C$ then ${\rm Q}f = f$ (in this case one has $p_X={\rm Id}_X$ and $p_Y={\rm Id}_Y$).
Define $\overline{{\rm Q}}(f)=L_{\mathcal C}({\rm Q}f)$.
By Lemma \ref{faithful}(1), $L_{\mathcal C}({\rm Q}f)$ does not depend on the choice of ${\rm Q}f$.
Thus $\overline{{\rm Q}}$ is well-defined, taking weak equivalences into isomorphisms, and the restriction $\overline{{\rm Q}}|_{\mathcal C}$
is exactly $L_\mathcal C$. Therefore
$$L_{\mathcal C} = \overline{{\rm Q}}\circ i_\mathcal M$$ where $i_\mathcal M: \mathcal C\longrightarrow \mathcal M$ is the embedding.

\vskip5pt

Similarly, there is a well-define functor $\overline{{\rm R}}_\mathcal C:\mathcal C\longrightarrow \Ho(\mathcal C\cap \mathcal F)$, taking weak equivalences into isomorphisms,
and the restriction $\overline{{\rm R}}_\mathcal C|_{\mathcal C\cap \mathcal F}$
is exactly $L_{\mathcal C\cap \mathcal F}$. Thus, one has
$$L_{\mathcal C\cap \mathcal F} = \overline{{\rm R}}_\mathcal C\circ i_\mathcal C$$ where $i_\mathcal C: \mathcal C\cap \mathcal F \longrightarrow \mathcal C$ is the embedding.

\vskip5pt

By the universal property of localization functors, one knows that $\overline{{\rm R}}_\mathcal C$ factors through $L_{\mathcal C}$ by
$\overline{{\rm R}}_\mathcal C': \Ho(\mathcal C)\longrightarrow \Ho(\mathcal C\cap \mathcal F)$,  and that
$\overline{{\rm R}}_{\mathcal C}'\circ\overline{{\rm Q}}$ factors through $L$ by
$\overline{p}: \Ho(\mathcal M)\longrightarrow \Ho(\mathcal C\cap \mathcal F)$:

\[
\xymatrix{
\mathcal C\ar[r]^-{L_\mathcal C}\ar[d]_{\overline{\rm R}_{\mathcal C}}& \Ho(\mathcal C)\ar[dl]^-{\overline{{\rm R}}_{\mathcal C}'}
&& M \ar[r]^-{L}\ar[d]_{\overline {{\rm R}}_{\mathcal C}'\circ \overline {{\rm Q}}}& \Ho(\mathcal M)\ar[dl]^-{\overline{p}}\\
\Ho(\mathcal C\cap \mathcal F) & && \Ho(\mathcal C\cap \mathcal F).
}
\]
Now there is a commutative diagram of functors
\[
\xymatrix{
\mathcal C\cap \mathcal F \ar[r]^-i\ar[d]_{L_{\mathcal C\cap \mathcal F}}& M\ar[dl]^-{\overline{{\rm R}}_{\mathcal C}'\circ\overline{Q}}\\
\Ho(\mathcal C\cap \mathcal F) &
}
\]
In fact, one has
$$
\overline{{\rm R}}_{\mathcal C}'\circ\overline{Q}\circ i=\overline{{\rm R}}_{\mathcal C}'\circ\overline{Q}\circ i_\mathcal M\circ i_\mathcal C=\overline{{\rm R}}_{\mathcal C}'\circ L_\mathcal C\circ i_\mathcal C=\overline{{\rm R}}_{\mathcal C}\circ i_\mathcal C=L_{\mathcal C\cap \mathcal F}.$$
Thus, by $(3.1)$ one has $$L_{\mathcal C\cap \mathcal F}=\overline{{\rm R}}_{\mathcal C}'\circ\overline{Q}\circ i =\overline{p}\circ L\circ i=\overline{p}\circ\overline{i}\circ L_{\mathcal C\cap \mathcal F}.$$
Therefore $\overline{p}\circ\overline{i}={\rm Id}_{\Ho(\mathcal C\cap \mathcal F)}$, by the universal property of localization functors.
In particular, $\overline{i}$ is faithful.

\vskip5pt

It remains to prove that $\overline{i}$ is full. That is, for any morphism  $\delta: X\longrightarrow Y\in {\rm Mor}(\Ho(\mathcal M))$ with $X$ and $Y$ in $\mathcal C\cap \mathcal F$, we need to prove $\delta = \overline{i}(\psi)$ for some $\psi: X\longrightarrow Y$ in ${\rm Mor}(\Ho(\mathcal C\cap \mathcal F))$. Note that $\delta$ can be written as $\delta={g_n}^{-1}f_n\cdots {g_1}^{-1}f_1$:
\[
\xymatrix{
X\ar[r]^-{f_1}& A_1& B_1\ar[l]_-{g_1}\ar[r]^-{f_2} & A_2& B_2\ar[l]_-{g_2}\ar[r]^-{f_3}& \ldots \ar[r]^-{f_{n}}& A_{n}& B_n=Y\ar[l]_-{g_{n}}
}
\]
where $g_1,\ldots, g_n\in \Weq$. See Subsection 2.3.

\vskip5pt

The main idea of the proof  below is to replace $A_j$ and $B_j$, by some objects in $\mathcal C\cap \mathcal F$, by using the factorizations $(3.2)$ and $(3.3)$, as well as Factorization axiom.
For each object $A_j$ we use the factorization $(3.3)$, and for each object $B_j$ we use the factorization $(3.2)$:
\[
\xymatrix@R=0.5cm{A_j \ar[rr]\ar[rd]_-{i_{_{A_j}}} && 0 && 0\ar[rr]\ar[rd]&& B_j\\
&{\rm R}A_j\ar[ru] && && {\rm Q}B_j\ar[ru]_-{p_{_{B_j}}}
}
\]
with ${\rm R}A_j\in \mathcal F$, \ $i_{_{A_j}}\in \TCoFib$, and  ${\rm Q}B_j\in \mathcal C$, \ $p_{_{B_j}}\in \TFib$.
Then one has the sequence of morphisms
\[
\xymatrix{
X\ar[r]^-{i_{_{A_1}}f_1}& {\rm R}A_1& {\rm Q}B_1\ar[l]_-{i_{_{A_1}}g_1p_{_{B_1}}}\ar[r]^-{i_{_{A_2}}f_2p_{_{B_1}}} & {\rm R}A_2& {\rm Q}B_2
\ar[l]_-{i_{_{A_2}}g_2p_{_{B_2}}}\ar[r]^-{i_{_{A_3}}f_3p_{_{B_2}}}& \ldots \ar[rr]^-{i_{_{A_n}}f_{n}p_{_{B_{n-1}}}}&& {\rm R}A_{n}& Y\ar[l]_-{i_{_{A_n}}g_{n}}
}
\]
and  $$\delta=(i_{_{A_n}}g_n)^{-1} (i_{_{A_n}}f_{n}p_{_{B_{n-1}}})\cdots (i_{_{A_2}}g_2p_{_{B_2}})^{-1} (i_{A_2}f_2p_{_{B_1}})(i_{_{A_1}}g_1p_{_{B_1}})^{-1} (i_{_{A_1}}f_1).$$
Consider factorizations of weak equivalences $i_{_{A_j}}g_jp_{_{B_j}}: {\rm R}A_j\longrightarrow {\rm Q}B_j$ for $j=1, \ldots,n:$
\[
\xymatrix@R=0.5cm{
{\rm R}A_j&& {\rm Q}B_j\ar[ll]_-{i_{_{A_j}}g_jp_{_{B_j}}}\ar[ld]^-{\sigma_j}\\
& C_j\ar[lu]^{\pi_j}&
}
\]
with $\sigma_j\in \TCoFib$ and $\pi_j\in \TFib$.
Note that ${\rm Q}B_n=B_n=Y\in \mathcal F$ and $p_{_{B_n}}={\rm Id}_Y$ since $Y\in \mathcal C$.
Since $\sigma_j: {\rm Q}B_j\longrightarrow C_j$ is a cofibration and ${\rm Q}B_j\in\mathcal C$, by Lemma \ref{CCoFib}(1)  $C_j\in\mathcal C$.
Since $\pi_j: C_j\longrightarrow {\rm R}A_j$ is a fibration and ${\rm R}A_j\in\mathcal F$, by Lemma \ref{CCoFib}(3) $C_j\in\mathcal F$.
Thus each $C_j\in\mathcal C\cap\mathcal F$.

\vskip5pt

Since $\pi_1: C_1 \longrightarrow {\rm R}A_1$ is a trivial fibration and $X\in\mathcal C$, by
Lemma \ref{split}(4) there exists  $h_1: X\longrightarrow C_1$ such that $\pi_1 h_1=i_{_{A_1}}f_1: X\longrightarrow {\rm R}A_1.$ See diagram

\[
\xymatrix@R=0.5cm{
X\ar[rr]^-{i_{_{A_1}}f_1}\ar@{..>}[rd]_-{h_1}&& {\rm R}A_j\\
& C_j\ar[ru]_-{\pi_1}&
}
\]

For $j=2,\ldots, n$, since $\pi_j: C_j \longrightarrow {\rm R}A_j$ is a trivial fibration and ${\rm Q}B_{j-1}\in\mathcal C$, by
Lemma \ref{split}(4) there exists
$h_j: {\rm Q}B_{j-1}\longrightarrow C_j$ such that $\pi_j h_j=i_{_{A_j}}f_jp_{_{B_{j-1}}}: {\rm Q}B_{j-1}\longrightarrow {\rm R}A_j$.  See diagram

\[
\xymatrix@R=0.5cm{
{\rm Q}B_{j-1}\ar[rr]^-{i_{_{A_j}}f_jp_{_{B_{j-1}}}}\ar@{..>}[rd]_-{h_j}&& {\rm R}A_j\\
& C_j\ar[ru]_-{\pi_j}&
}
\]

\vskip5pt

Again,  Since $\sigma_{j-1}:  {\rm Q}B_{j-1}\longrightarrow C_{j-1}$ is a trivial cofibration and $C_j\in\mathcal F$, by  Lemma \ref{split}(2) there exists $r_j:C_{j-1}\longrightarrow C_j$ such that $r_j\sigma_{j-1}=h_j$ for $j=2,\ldots, n$. See diagram

\[
\xymatrix@R=0.5cm{{\rm Q}B_{j-1}\ar[rr]^-{\sigma_{j-1}}
\ar[rd]_-{h_j}&& C_{j-1}\ar@{..>}[ld]^-{r_j}\\
& C_j &
}
\]
Thus, one has $i_{_{A_j}}f_jp_{_{B_{j-1}}} = \pi_jh_j = \pi_jr_j\sigma_{j-1}$ for $j=2,\ldots, n$.

\vskip5pt

In this way we obtain a sequence of morphisms in $\mathcal C\cap \mathcal F$:
\[
\xymatrix{
X\ar[r]^-{h_1}&C_1\ar[r]^-{r_2}&C_2\ar[r]^-{r_3}&\cdots\ar[r]^-{r_n}&C_n&Y.\ar[l]_-{\sigma_n}
}
\]
Put $\psi=(\sigma_n)^{-1}r_n\cdots r_2h_1 \in {\rm Mor}(\Ho(\mathcal C\cap \mathcal F))$. We claim $\delta = \overline{i}(\psi)$. In fact, in $\Ho(\mathcal M)$ one has

\[
\begin{aligned}
\delta&=(i_{_{A_n}}g_n)^{-1} (i_{_{A_n}}f_{n}p_{_{B_{n-1}}})\cdots (i_{_{A_2}}g_2p_{_{B_2}})^{-1} (i_{A_2}f_2p_{_{B_1}})(i_{_{A_1}}g_1p_{_{B_1}})^{-1} (i_{_{A_1}}f_1)\\
&=(\pi_n\sigma_n)^{-1}(\pi_nr_n\sigma_{n-1}) \cdots(\pi_2\sigma_2)^{-1}(\pi_2r_2\sigma_1)(\pi_1\sigma_1)^{-1}(\pi_1h_1)\\
&=(\sigma_n)^{-1}r_n\cdots r_2h_1\\
&=((L\circ i)(\sigma_n))^{-1}\circ (L\circ i)(r_n\cdots r_2h_1)\\
&=((\overline{i}\circ L_{\mathcal C\cap \mathcal F})(\sigma_n))^{-1}\circ (\overline{i}\circ L_{\mathcal C\cap \mathcal F})(r_n\cdots r_2h_1)\\
&=(\overline{i}(L_{\mathcal C\cap \mathcal F}(\sigma_n))^{-1})\circ \overline{i}(L_{\mathcal C\cap \mathcal F}(r_n\cdots r_2h_1))\\
&=\overline{i}((L_{\mathcal C\cap \mathcal F}(\sigma_n))^{-1}\circ L_{\mathcal C\cap \mathcal F}(r_n\cdots r_2h_1))\\
&=\overline{i}((\sigma_n)^{-1}r_n\cdots r_2h_1)\\
&=\overline{i}(\psi).
\end{aligned}
\]
This completes the proof. \hfill $\square$

\vskip5pt

\subsection{\bf Proof of Theorem \ref{hoA}} \ By Lemma \ref{object}, $\mathcal C\cap \mathcal F$ is an additive category and
$\mathcal C\cap \mathcal F\cap \mathcal W$ is a full additive subcategory. Thus $(\mathcal C\cap \mathcal F)/(\mathcal C\cap \mathcal F\cap \mathcal W)$ is an additive category.
Set
$$\mathcal S =\{f\in {\rm Mor}(\mathcal C\cap \mathcal F) \ |\ \overline{f} \text{\ is\ an\ isomorphism\ in\ } (\mathcal C\cap \mathcal F)/(\mathcal C\cap \mathcal F\cap \mathcal W)\}.$$
Then $(\mathcal C\cap \mathcal F)/(\mathcal C\cap \mathcal F\cap \mathcal W)\cong(\mathcal C\cap \mathcal F)[\mathcal S^{-1}]$ (cf. Lemma \ref{locilzationasquotient}).
By Theorem \ref{HoCcapF} one has $$\Ho(\mathcal A) \cong \Ho(\mathcal C\cap \mathcal F)=(\mathcal C\cap \mathcal F)[(\Weq\cap {\rm Mor}(\mathcal C\cap \mathcal F))^{-1}].$$ Thus, what we need to prove is
exactly $$(\mathcal C\cap \mathcal F)[\mathcal S^{-1}]\cong (\mathcal C\cap \mathcal F)[(\Weq\cap {\rm Mor}(\mathcal C\cap \mathcal F))^{-1}].$$
So, it suffices to prove that for $f:A\longrightarrow B$ with  $A\in \mathcal C\cap \mathcal F$ and $B\in \mathcal C\cap \mathcal F$, $f\in \Weq$ if and only if $f\in \mathcal S$.

\vskip5pt

Suppose that $f\in \Weq$. Then there is a factorization $f=tu$ with $u\in \TCoFib$ and $t\in \TFib$.
\[
\xymatrix@R=0.5cm{
A\ar[rr]^-{f}\ar@<0.5ex>[dr]_-{u} &  & B\\
& C\ar@<0.5ex>[ur]_-{t} &
}
\]
Then $u$ is a splitting monomorphism, by Lemma \ref{split}(2). Since $\mathcal A$ is a weakly idempotent complete additive category, $\Coker u$ exists. By Lemma \ref{morphism}(4),
$\Coker u\in {\rm T}\mathcal C$. Similarly, $t$ is a splitting epimorphism with $\Ker t\in {\rm T}\mathcal F$.
Since $0\longrightarrow C = u\circ (0\longrightarrow A)$, it follows that $0\longrightarrow C$ is a cofibration and hence $C\in \mathcal C$.
Thus $\Ker t\in \mathcal C$ as a direct summand of $C$, and hence $\Ker t\in \mathcal C\cap \mathcal F\cap \mathcal W$. Similarly, $\Coker u\in \mathcal C\cap \mathcal F\cap \mathcal W$.

\vskip5pt

Now we have two splitting short exact sequences
$$0\longrightarrow A \stackrel u \longrightarrow C \stackrel {c}\longrightarrow \Coker u \longrightarrow 0, \ \ \
0\longrightarrow \Ker t \stackrel k \longrightarrow C \stackrel {t}\longrightarrow B \longrightarrow 0.$$
Thus, there are morphisms \ $v: C\longrightarrow A$, \ $e: \Coker u\longrightarrow C$,  \ $s:B\longrightarrow C$ and  \ $p: C\longrightarrow \Ker t$,
such that $$v\circ u= {\rm Id}_{A}, \ \ \ c\circ e= {\rm Id}_{\Coker u}, \ \ \ u\circ v+e \circ c= {\rm Id}_C$$
and that
$$t\circ s= {\rm Id}_B,  \ \ \ p\circ k= {\rm Id}_{\Ker t}, \ \ \  k\circ p + s\circ t= {\rm Id}_C.$$ Put \ $g=vs$.
Then $${\rm Id}_A - gf=vu-vstu=v({\rm Id}_C- s t)u=vkpu$$ and $${\rm Id}_B-fg=ts-tuvs=t({\rm Id}_C-uv)s=tecs.$$
Since ${\rm Id}_A - gf = (v\circ k)\circ (p\circ u)$ factors through $\Ker t\in \mathcal C\cap \mathcal F\cap \mathcal W$ and
${\rm Id}_B-fg = (t\circ e)\circ (c\circ s)$ factors through $\Coker u\in \mathcal C\cap \mathcal F\cap \mathcal W$, it follows that
$\overline{f}$ is an isomorphism in $(\mathcal C\cap \mathcal F)/(\mathcal C\cap \mathcal F\cap \mathcal W)$, i.e., $f\in \mathcal S$.

\vskip5pt

Conversely, let $f\in \mathcal S$. By definition there is a morphism $g:B\longrightarrow A$ such that ${\rm Id}_A - gf =ts$ for some  $s:A\longrightarrow U$ and $t: U\longrightarrow A$ with $U\in \mathcal C\cap \mathcal F\cap \mathcal W$.
Thus  $(g, t)\left(\begin{smallmatrix} f\\s\end{smallmatrix}\right)={\rm Id}_A$. In particular, $\left(\begin{smallmatrix} f\\s\end{smallmatrix}\right): A \longrightarrow B\oplus U$ is a splitting monomorphism. By assumption $\left(\begin{smallmatrix} f\\s\end{smallmatrix}\right)$ has cokernel. It follows that
there is a splitting short exact sequence $$0\longrightarrow A \stackrel{\left(\begin{smallmatrix} f\\s\end{smallmatrix}\right)}\longrightarrow B\oplus U \stackrel {(b, u)} \longrightarrow \Coker \left(\begin{smallmatrix} f\\s\end{smallmatrix}\right) \longrightarrow 0.$$
Since $U\in \mathcal C\cap \mathcal F\cap \mathcal W$ and $\overline{f}$ is an isomorphism in $(\mathcal C\cap \mathcal F)/(\mathcal C\cap \mathcal F\cap \mathcal W)$, it follows that $\overline{\left(\begin{smallmatrix} f\\s\end{smallmatrix}\right)}: A\longrightarrow B\oplus U$ is also an isomorphism in $(\mathcal C\cap \mathcal F)/(\mathcal C\cap \mathcal F\cap \mathcal W)$. Thus $\overline {(b, u)} = 0$ in  $(\mathcal C\cap \mathcal F)/(\mathcal C\cap \mathcal F\cap \mathcal W)$. But $(b, u)$ is a splitting epimorphism, thus there is a morphism $\left(\begin{smallmatrix} b'\\u'\end{smallmatrix}\right): \Coker \left(\begin{smallmatrix} f\\s\end{smallmatrix}\right) \longrightarrow B\oplus U$ such that $(b,u)\left(\begin{smallmatrix} b'\\u'\end{smallmatrix}\right) = {\rm Id}_{\Coker \left(\begin{smallmatrix} f\\s\end{smallmatrix}\right) }$.
Therefore in $(\mathcal C\cap \mathcal F)/(\mathcal C\cap \mathcal F\cap \mathcal W)$ there holds
$${\rm Id}_{\Coker \left(\begin{smallmatrix} f\\s\end{smallmatrix}\right) } = \overline{{\rm Id}_{\Coker \left(\begin{smallmatrix} f\\s\end{smallmatrix}\right)} } =
\overline{(b,u)}\overline{\left(\begin{smallmatrix} b'\\u'\end{smallmatrix}\right)} = 0.$$
Thus $\Coker \left(\begin{smallmatrix} f\\s\end{smallmatrix}\right) \in \mathcal C\cap \mathcal F\cap \mathcal W$. In particular,
$\Coker \left(\begin{smallmatrix} f\\s\end{smallmatrix}\right) \in \mathcal C\cap \mathcal W$.
By Lemma \ref{morphism}(6), $\left(\begin{smallmatrix} f\\s\end{smallmatrix}\right)\in \TCoFib$.

\vskip5pt

Also, $(1,0): A\oplus U\longrightarrow A$ is a splitting epimorphism with kernel $U\in \mathcal F\cap \mathcal W$.
By Lemma \ref{morphism}(7), $(1, 0)\in \TFib$.
Thus, by the decomposition $f=(1, 0)\left(\begin{smallmatrix} f\\s\end{smallmatrix}\right)$ one sees that $f\in \Weq.$ \hfill $\square$

\vskip5pt

\section{\bf Construction of fibrant model structures}

The aim of this section is to prove Theorems \ref{fibrant}, \ref{mainthm} and \ref{correspondence}, which describe and construct fibrant model structures.
Throughout this section, $\mathcal{A}$ is a weakly idempotent complete additive category.

\subsection{Proof of Theorem \ref{fibrant}}

\ $(1) \Longrightarrow (2):$ \ Suppose that $\mathcal F=\mathcal A$.
If $f$ is a splitting monomorphism with $\Coker f\in \mathcal C$, then $\Coker f\in {\rm T}\mathcal C$.
By Lemma \ref{morphism}(6) one sees that $f\in \TCoFib$.
Conversely, if $f\in \TCoFib$,  then $f$ is a splitting monomorphism, by Lemma \ref{split}(2);  and $\Coker f\in {\rm T}\mathcal C$, by Lemma \ref{morphism}(4).

\vskip5pt

$(2) \Longrightarrow (3):$  \   Suppose that
$\TCoFib=\{\text{splitting monomorphism} \ f \ |\ \Coker f\in {\rm T}\mathcal C \}$.
By Lemma \ref{split}(3) one sees the inclusion $\Fib \subseteq \{ \text{morphism} \ g \ | \ g \ \mbox{is} \ \Hom_{\mathcal A}({\rm T}\mathcal C, -)\mbox{-epic}\}$.
Conversely, let $g: C\longrightarrow D$ be a $\Hom_{\mathcal A}({\rm T}\mathcal C, -)$-epimorphism. Then for any
trivial cofibration $f: A\longrightarrow B$, by assumption $f$ is a splitting monomorphism with  $\Coker f\in {\rm T}\mathcal C$.
Thus one has a splitting short exact sequence $0\longrightarrow A\stackrel f \longrightarrow B\stackrel p \longrightarrow \Coker f\longrightarrow 0$
with morphisms  $f': B\longrightarrow A$ and $p': \Coker f\longrightarrow B$ such that
$$f'f={\rm Id}_A,  \ \ pp'= {\rm Id}_C, \ \ ff'+p'p={\rm Id}_B.$$
Given an arbitrary commutative square
\[
\xymatrix{
A\ar[r]^-a\ar[d]_-f & C\ar[d]^-g\\
B\ar[r]^-{b}& D
}
\]
one has $bp': \Coker f \longrightarrow D$. Since $\Coker f\in {\rm T}\mathcal C$ and $g: C\longrightarrow D$ is a $\Hom_{\mathcal A}({\rm T}\mathcal C, -)$-epimorphism,
it follows that there exists a morphism $h: \Coker f \longrightarrow C$ such that $bp' = gh.$
Thus the commutative square has a lifting $t = af' + hp: B\longrightarrow C$. Using $\Fib={\rm RLP}(\TCoFib)$ one sees that $g\in \Fib$.
This proves $\Fib  = \{ \text{morphism} \ f \ | \ f \ \mbox{is} \ \Hom_{\mathcal A}({\rm T}\mathcal C, -)\mbox{-epic}\}$.

\vskip5pt

$(3) \Longrightarrow (1):$ \ Suppose that $\Fib$ is exactly the class of $\Hom_{\mathcal A}({\rm T}\mathcal C, -)$-epic morphisms.
Since
 $X\longrightarrow 0$ is $\Hom_{\mathcal A}({\rm T}\mathcal C, -)$-epic for any object $X$ of $\mathcal A$, it follows that $X\in\mathcal F$. Thus $\mathcal F=\mathcal A$.

\vskip5pt

The last asserion follows from Theorem \ref{hoA}. \hfill $\square$

\begin{prop}\label{fibrantFWFS} \ Let $(\CoFib, \Fib, \Weq)$ be a fibrant model structure. Then both $(\CoFib, \TFib)$ and $(\TCoFib, \Fib)$ are fibrantly weak factorization systems.
\end{prop}

\begin{proof} \ One needs to check that $(\CoFib, \TFib)$ satisfies the conditions in Definition \ref{FWFS}.
It is clear that $(\CoFib, \TFib)$ is a weak factorization system.
For $f: A\longrightarrow B$ and $g:B\longrightarrow C$,
suppose that  $f$ and $gf$ are in $\TFib$.
Then $g\in \Weq$, by Two out of three axiom.
By Theorem \ref{fibrant}, $\Fib$ is exactly the class of morphisms which are $\Hom_{\mathcal M}({\rm T}\mathcal C, -)$-epic.
Since $gf$ is $\Hom_{\mathcal{A}}({\rm T}\mathcal C,-)$-epic, so is $g$. Thus $g\in \Fib \cap \Weq =\TFib$. Finally, ${\rm T}\mathcal C$ is contravariantly finite, by Lemma \ref{object}(1).
\end{proof}

\subsection{\bf Properties of fibrantly weak factorization systems}
Throughout this subsection, we fix the following assumption and notations, without repeating them each time.
Assume that $(\CoFib, \TFib)$ is a fibrantly weak factorization system on $\mathcal A$. We remind the reader, although we use the notations
$\CoFib$ and $\TFib$, however, at this moment they are not assumed to be from a model structure.
Put
$\mathcal C, \ \mathcal W, \ {\rm T}\mathcal C, \ \Fib, \ \TCoFib$ and \ $\Weq$ as in $(1.1)$. Namely,
\begin{align*}
\mathcal C & = \{ \mbox{object} \ X \ | \ \mbox{morphism} \ 0\longrightarrow X \ \mbox{is in} \ \CoFib\} \\
\mathcal W & = \{ \mbox{object} \ X \ | \ \mbox{morphism}  \ X\longrightarrow 0 \ \mbox{is in} \ \TFib\} \\
{\rm T}\mathcal C & = \mathcal C\cap \mathcal W \\
\Fib & = \{ \text{morphism} \ f \ | \ f \ \mbox{is} \ \Hom_{\mathcal A}({\rm T}\mathcal C, -)\mbox{-epic}\}\\
\TCoFib & =\{\text{splitting monomorphism} \ f \ |\ \Coker f\in {\rm T}\mathcal C \}\\
\Weq & =\TFib\circ \TCoFib.\end{align*}

\begin{lem}\label{summand} \ With the same assumption and notations as above. Then the classes $\mathcal C$,  $\mathcal W$, and ${\rm T}\mathcal C$ are closed under direct summands.
\end{lem}
\begin{proof} \ By definition both $\CoFib$ and $\TFib$ are closed under retracts. It follows that $\mathcal C$ and $\mathcal W$ are closed under direct summands,
and hence so is ${\rm T}\mathcal C = \mathcal C\cap \mathcal W.$
\end{proof}

\begin{lem}\label{intersection1} \ With the same assumption and notations as above. Then  \ $\TCoFib= \CoFib\cap \Weq.$
\end{lem}
\begin{proof} \ The inclusion $\TCoFib\subseteq \CoFib$ follows by Lemma \ref{lemFWFS}$(5)$. For any morphism $f: A\longrightarrow B$ in $\TCoFib$,
$f = {\rm Id}_B \circ f$, where ${\rm Id}_B\in \TFib$ by Lemma \ref{lemFWFS}$(1)$. Thus $f\in \TFib\circ \TCoFib = \Weq$. This proves $\TCoFib \subseteq \CoFib\cap \Weq$.

\vskip5pt

Conversely, let $f:A\longrightarrow B$ be a morphism in $\CoFib\cap \Weq$. Since $\Weq = \TFib\circ \TCoFib$,  one has factorization
$f=hg$, where $g: A\longrightarrow C$ is in $\TCoFib$ and $h: C\longrightarrow B$ is in $\TFib$.
Then  commutative diagram
  \[
  \xymatrix{
  A\ar[r]^g\ar[d]_-f&C\ar[d]^-{h}\\
  B\ar@{=}[r]\ar@{..>}[ur]^-u&B
  }
  \]
gives a lifting $u: B\longrightarrow C$ such that $uf=g$ and $hu={\rm Id}_B$.
Since $g\in \TCoFib$, by definition $g$ is a splitting monomorphism with $\Coker g\in {\rm T}\mathcal C$. By $uf = g$ one sees that $f$
is a splitting monomorphism. Then one has commutative diagram with two splitting short exact sequences
 \[
  \xymatrix{
  0\ar[r]& A\ar@{=}[d]\ar[r]^-{f}&B\ar[d]^-u\ar[r]^-{p}\ar[r]&\Coker f\ar@{..>}[d]^-t\ar[r]&0\\
  0\ar[r]& A\ar[r]^-{g}&C\ar[r]^-c\ar[r]&\Coker f\ar[r]&0
  }
  \]
where $u$ is  a splitting monomorphism. Thus $t$ a splitting monomorphism. Therefore $\Coker f$ is a direct summand of $\Coker g$.
By Lemma \ref{summand}, $\Coker f\in {\rm T}\mathcal C$. By definition $f\in \TCoFib$. This proves $\CoFib\cap \Weq\subseteq \TCoFib$, and hence
$\TCoFib= \CoFib\cap \Weq$. \end{proof}

\begin{lem}\label{lifting} \ With the same assumption and notations as above. Let
\[
\xymatrix@R=0.5cm{
A\ar[r]^-{f}\ar[d]_-{i} & C\ar[d]^-{p}\\
B\ar[r]^-{g} & D
}
\]
be a commutative square with $i\in \TCoFib$ and $p\in \Fib$. Then there is a lifting $t: B\longrightarrow C$ such that $ti = f$ and $pt = g$.
\end{lem}
\begin{proof} \  By definition $i$ is a splitting monomorphism with $\Coker i\in {\rm T}\mathcal C$,
and hence there is
a splitting exact sequence $0\longrightarrow A\stackrel i\longrightarrow B\stackrel c \longrightarrow \Coker i\longrightarrow 0$. Therefore there are morphisms $i': B\longrightarrow A$ and  $c':\Coker i \longrightarrow B$ such that
 $$i'i = {\rm Id}_A, \ \ \ cc'= {\rm Id}_{\Coker i}, \ \ \ ii'+ c'c = {\rm Id}_B.$$
 Then one has a morphism $gc': \Coker i \longrightarrow D.$
Since $\Fib = \{ \text{morphism} \ f \ | \ f \ \mbox{is} \ \Hom_{\mathcal A}({\rm T}\mathcal C, -)\mbox{-epic}\}$,
it follows that $\Hom_{\mathcal{A}}(\Coker i, p)$ is epic. Thus,  there is a morphism $s:\Coker i \longrightarrow C$ such that $gc'= ps$. Then there is a lifting
$$t= fi'+sc: B\longrightarrow C$$ such that $ti = f$ and $pt = g$.\end{proof}

\begin{lem}\label{intersection2} \ With the same assumption and notations as above. Then  \ $\TFib=\Fib\cap \Weq.$
\end{lem}
\begin{proof} \  The inclusion $\TFib\subseteq \Weq$ follows from  the definition $\Weq = \TFib\circ \TCoFib$, since the identity morphisms are in
$\TCoFib$.
To see the inclusion $\TFib\subseteq \Fib$,  for any $f: A \longrightarrow B$ in $\TFib$, for any object $X\in {\rm T}\mathcal C$,
the morphism $0: 0 \longrightarrow X$ is in $\TCoFib \subseteq \CoFib$.  Then any commutative diagram
  \[\xymatrix{0\ar[r]\ar[d] &A\ar[d]^-{f}\\
  X\ar[r]\ar@{..>}[ur]^-u&B
  }
  \]
will give a lifting, which means that $f$ is $\Hom_\mathcal A(X, -)$-epic. By definition $f\in \Fib$. Thus $\TFib \subseteq \Fib\cap \Weq$.

\vskip5pt

Conversely, suppose that $f:A\longrightarrow B$ is in $\Fib\cap \Weq$. Decompose $f=hg$, where $g:A\longrightarrow C\in \TCoFib$ and $h: C\longrightarrow B\in \TFib$.
Then one has commutative diagram
  \[\xymatrix{A\ar@{=}[r]\ar[d]_-g &A\ar[d]^-{f}\\
  C\ar[r]^-h\ar@{..>}[ur]^-u&B
  }
  \]
with $g\in \TCoFib$ and $f\in \Fib$. By Lemma \ref{lifting} there is
a lifting $u: C\longrightarrow A$ such that $ug={\rm Id}_A$ and $fu=h$. The commutative diagram
\[
\xymatrix{
A\ar[r]^-g\ar[d]_-f &C\ar[r]^-u\ar[d]_-h & A\ar[d]^-f  \\
B\ar@{=}[r]&B\ar@{=}[r]&B
}
\]
shows that $f$ is a retract of $h$. Since $(\CoFib, \TFib)$ is a weak factorization system, by definition  $\TFib$ is closed under retracts.
It follows that $f\in \TFib$.
\end{proof}

\begin{lem}\label{TCoFibTFib} \ \ With the same assumption and notations as above.  For $f: A\longrightarrow B$ and $g:B\longrightarrow C$, if $f\in \TCoFib$ and $gf\in \TFib$, then $g\in \TFib$.
\end{lem}
\begin{proof} \ By Lemma \ref{intersection1}, $f\in \TCoFib \subseteq \CoFib$. The commutative diagram
  \[
  \xymatrix{
  A\ar@{=}[r]\ar[d]_-f&A\ar[d]^-{gf}\\
  B\ar[r]^-g\ar@{..>}[ur]^-u&C
  }
  \]
  gives a lifting $u$ such that $uf={\rm Id}_A$ and $gfu=g$. Thus $u$ is a splitting epimorphism.
Since by assumption  $f$ is a splitting monomorphism with $\Coker f\in {\rm T}\mathcal C$ and $\Ker u = \Coker f$ (cf. Definition-Fact \ref{defnfact}), 
it follows that $\Ker u\in {\rm T}\mathcal C$, and hence $\Ker u\longrightarrow 0$ is in $\TFib$, by definition.
It follows from Lemma \ref{lemFWFS}(5) that  $u\in \TFib$. Then $g = (gf)\circ u\in \TFib$,  by Lemma \ref{lemFWFS}(2).
\end{proof}

\begin{lem}\label{Weq} \ With the same assumption and notations as above.  If $f: A\longrightarrow B$ is in $\Weq$, then for any right ${\rm T}\mathcal C$-approximation $t:U\longrightarrow B$,
morphism $(f,t):A\oplus U\longrightarrow B$ is in $\TFib$.
\end{lem}
\begin{proof} \ Since $f\in \Weq$, by definition one has a factorization $f=hg$
\[
\xymatrix@R=0.5cm{
A\ar[rr]^-{f}\ar@<0.5ex>[dr]_-{g} &  & B\\
& C\ar@<0.5ex>[ur]_-{h} &
}
\]
with $g\in \TCoFib$ and $h\in \TFib$. Thus $g$ is a splitting monomorphism  with $\Coker g\in {\rm T}\mathcal C$, and hence one has
a splitting exact sequence $\xymatrix{0\ar[r]& A\ar[r]^-{g}& C\ar[r]^-{p}\ar[r]&\Coker g\ar[r]&0}$ and a morphism
 $p':\Coker g \longrightarrow C$ with $pp'= {\rm Id}_{\Coker g}.$ Consider morphism $hp': \Coker g \longrightarrow B$.
 Since $t: U\longrightarrow B$ is a right ${\rm T}\mathcal C$-approximation and $\Coker g\in {\rm T}\mathcal C$,
there is a morphism $t': \Coker g\longrightarrow U$ such that $hp'=tt'$.

\vskip5pt
Since $(g, p'): A\oplus \Coker g \longrightarrow C$ is an isomorphism, it is in $\TFib$. Note that $\TFib$ is closed under compositions, by Lemma \ref{lemFWFS}(2). 
Thus $(f, hp') = h(g, p'): A\oplus \Coker g \longrightarrow B$ is in $\TFib$.
 Note that in commutative diagram
\[
\xymatrix{
& A\oplus \Coker g\ar[dl]_-{\left(\begin{smallmatrix}1 & 0 \\0 & 1 \\0 & 0 \end{smallmatrix}\right)}\ar[d]^-{(f, hp')}\\
A\oplus \Coker g\oplus U\ar[r]^-{(f, hp',t)}&B
}
\]
$\left(\begin{smallmatrix}1 & 0 \\0 & 1 \\0 & 0 \end{smallmatrix}\right)$ is a splitting monomorphism with cokernel $U\in {\rm T}\mathcal C$, i.e.,
$\left(\begin{smallmatrix}1 & 0 \\0 & 1 \\0 & 0 \end{smallmatrix}\right)\in {\rm TCoFib}.$
Thus $(f,hp', t)\in \TFib$, by Lemma \ref{TCoFibTFib}. Commutative diagram
\[
\xymatrix{
A\oplus U\ar[r]^-{\left(\begin{smallmatrix}1 & 0 \\0 & 0 \\0 & 1 \end{smallmatrix}\right)}\ar[d]_-{(f,t)} &A\oplus \Coker g\oplus U\ar[r]^-{\left(\begin{smallmatrix}1 & 0 &0 \\0 & t' & 1  \end{smallmatrix}\right)} \ar[d]_-{(f,hp', t)} & A\oplus U\ar[d]^-{(f,t)}  \\
B\ar@{=}[r]&B\ar@{=}[r]&B
}
\]
shows that $(f,t)$ is a retract of $(f, hp',t)$. Thus $(f,t)\in \TFib$.
\end{proof}

\begin{lem}\label{third2outof3} \ With the same assumption and notations as above.
Let $u: X\longrightarrow Y$ and $v:Y\longrightarrow Z$. Suppose that $v, vu\in \TFib$ and $u\in \CoFib$, then $u\in \TCoFib$.
\end{lem}
\begin{proof} \ The commutative diagram
  \[
  \xymatrix{
  X\ar@{=}[r]\ar[d]_-u&X\ar[d]^-{vu}\\
  Y\ar[r]^-v\ar@{..>}[ur]^-s&Z
  }
  \]
  gives a lifting $s$ such that $su={\rm Id}_X$ and $vus=v$. Thus $u$ is a splitting monomorphism. Then there are morphisms
  $\xymatrix{Y\ar@<0.5ex>[r]^-c &\Coker u\ar@<0.5ex>[l]^-{c'}}$ such that $cc'={\rm Id}_{\Coker u}$, \ $c u=0$, \ $sc'=0$ and $c'c+us={\rm Id}_Y$. Note that $0\longrightarrow \Coker u$ is in $\CoFib$ by Lemma \ref{lemFWFS}(3).
  Thus  commutative diagram
  \[
  \xymatrix{
  0\ar@{=}[r]\ar[d]&X\ar[d]^-{vu}\\
  \Coker u\ar[r]^-{vc'}\ar@{..>}[ur]^-h&Z
  }
  \]
  gives  a lifting $h: \Coker u\longrightarrow X$ such that $vuh=vc'$. To prove that $\Coker u\in {\rm T}\mathcal C$, it suffice to prove that $\Coker u \longrightarrow 0\in \TFib$. For $i:M\longrightarrow N\in \CoFib$ and $l: M\longrightarrow \Coker u$, the commutative diagram
\[
\xymatrix@R=0.5cm{
M\ar[r]^-{(c'-uh)l}\ar[d]_-{i} & Y\ar[d]^-{v}\\
N\ar[r]^-0\ar@{..>}[ur]^-\delta & Z
}
\]
gives a lifting $\delta: N\longrightarrow Y$ such that $\delta i=(c'-uh)l$. Then there is a commutative diagram
\[
\xymatrix@R=0.5cm{
M\ar[r]^-{l}\ar[d]_-{i} & \Coker u\ar[d]\\
N\ar[r]\ar@{..>}[ur]^-{c\delta}  & 0
}
\]
where $c\delta i= c(c'-uh)l=l$. Thus $\Coker u \longrightarrow 0$ is in $\TFib$, by Lemma \ref{lemFWFS}(1). Then $u\in \TCoFib$ by definition.
\end{proof}

\subsection{\bf Proof of Theorem \ref{mainthm}} \ Assume that $(\CoFib, \Fib, \Weq)$ is a fibrant model structure on $\mathcal A$.
Then by Proposition \ref{fibrantFWFS}, $(\CoFib, \TFib)$ is a fibrantly weak factorization system on $\mathcal A$; and by Theorem \ref{fibrant} one has
\begin{align*}\Fib &= \{ \text{morphism} \ f \ | \ f \ \mbox{is} \ \Hom_{\mathcal A}({\rm T}\mathcal C, -)\mbox{-epic}\}\\
\TCoFib &=\{\text{splitting monomorphism} \ f \ |\ \Coker f\in {\rm T}\mathcal C \}.
\end{align*}

Conversely, assume that $(\CoFib, \TFib)$ is a fibrantly weak factorization system on $\mathcal A$. We will prove that $(\CoFib, \Fib, \Weq)$ is a model structure on $\mathcal A$,
where $\Fib$ and $\Weq$ are given as in $(1.1)$.

\vskip5pt

{\bf Retract axiom} \ Since $(\CoFib, \TFib)$ is a fibrantly weak factorization system, by definition $\CoFib$ is closed under retracts.
Let $g$ be a retract of $f$, i.e., there is a commutative diagram
\[
\xymatrix@R=0.5cm{
    A'\ar[r]^{\varphi_1}\ar[d]_g & A\ar[r]^{\psi_1}\ar[d]_{f} & A'\ar[d]^-{g} \\
    B'\ar[r]^{\varphi_2} & B\ar[r]^{\psi_2} & B'
}
\]
such that   $\psi_1 \varphi_1 = {\rm Id}_{A'}$ and $\psi_2 \varphi_2 = {\rm Id}_{B'}$.

\vskip5pt

Suppose that $f\in \Fib$. For $U\in {\rm T}\mathcal C$, by definition $\Hom_{\mathcal{A}}(U, f)$ is epic.
Also $\Hom_{\mathcal{A}}(U, \psi_2)$ is epic since $\psi_2$ is a splitting epimorphism. Thus
by equality $$\Hom_{\mathcal{A}}(U, g)\circ\Hom_{\mathcal{A}}(U, \psi_1) = \Hom_{\mathcal{A}}(U, \psi_2)\circ\Hom_{\mathcal{A}}(U, f)$$ one sees that
$\Hom_{\mathcal{A}}(U, g)$ is epic. Thus $g\in \Fib$.

\vskip5pt

Suppose that $f\in \Weq$. Since $(\CoFib, \TFib)$ is a fibrantly weak factorization system, by definition
${\rm T}\mathcal C$ is contravariantly  finite in $\mathcal A$. Choose a right ${\rm T}\mathcal C$-approximation $t:U\longrightarrow B'$ and a right ${\rm T}\mathcal C$-approximation $s:V\longrightarrow B$.
Consider $\psi_2s: V\longrightarrow B'$. Then there is a morphism $h:V\longrightarrow U$ such that $\psi_2s=th$. Note that $(s, \varphi_2t):V\oplus U\longrightarrow B$ is also a right ${\rm T}\mathcal C$-approximation. Thus $(f,s,\varphi_2t): A\oplus V\oplus U\longrightarrow B\in \TFib$ by Lemma \ref{Weq}.  The following commutative diagram
\[
\xymatrix{
A'\oplus U\ar[r]^-{\left(\begin{smallmatrix}\varphi_1 & 0 \\0 & 0 \\0 & 1 \end{smallmatrix}\right)}\ar[d]_-{(g,t)} &A\oplus V\oplus U\ar[r]^-{\left(\begin{smallmatrix}\psi_1 & 0 &0 \\0 & h & 1  \end{smallmatrix}\right)} \ar[d]^-{(f,s,\varphi_2t)} & A'\oplus U\ar[d]^-{(g,t)}  \\
B'\ar[r]^-{\varphi_2}&B\ar[r]^-{\psi_2}&B'
}
\]
shows that $(g,t)$ is a retract of $(f,s,\varphi_2t)$. Since $(\CoFib, \TFib)$ is a weak factorization system, by definition $\TFib$ is closed under retracts.
Thus $(g,t)\in \TFib$. Since $\left(\begin{smallmatrix} 1 \\ 0\end{smallmatrix}\right): A\longrightarrow A\oplus U\in \TCoFib$
it follows that $g=(g,t)\left(\begin{smallmatrix} 1 \\ 0\end{smallmatrix}\right)\in \TFib\circ \TCoFib$.

\vskip10pt

{\bf Lifting axiom} \ \ Let
\[
\xymatrix@R=0.5cm{
A\ar[r]^-{f}\ar[d]_-{i} & C\ar[d]^-{p}\\
B\ar[r]^-{g} & D
}
\]
be a commutative square with $i\in \CoFib$ and $p\in \Fib$.

\vskip5pt

If $p\in \Weq$, then $p\in \Fib\cap \Weq = \TFib$, by Lemma \ref{intersection2}. Since $(\CoFib, \TFib)$ is a weak factorization system, by definition
there is a lifting $t: B\longrightarrow C$ such that $ti = f$ and $pt = g$.

\vskip5pt

If $i\in \Weq$, then $i\in \CoFib\cap \Weq = \TCoFib$, by Lemma \ref{intersection1}. By Lemma \ref{lifting}, there is a lifting
$t: B\longrightarrow C$ such that $ti = f$ and $pt = g$.

\vskip10pt

{\bf Factorization axiom} \ One only needs to prove that any morphism $f:A\longrightarrow B$ has a factorization $f=pi$ where $i\in \TCoFib$ and $p\in \Fib$. Since ${\rm T}\mathcal C$ is contravariantly finite, there is a right ${\rm T}\mathcal C$-approximation $t: U \longrightarrow B$. Note that $(f, t): A\oplus U \longrightarrow B$ is $\Hom_{\mathcal{A}}({\rm T}\mathcal C, -)$-epic. By definition $(f, t)\in \Fib.$ Since $\left(\begin{smallmatrix} 1 \\ 0\end{smallmatrix}\right): A\longrightarrow A\oplus U$ is in $\TCoFib$, it follows that $f=(f, t)\left(\begin{smallmatrix} 1 \\ 0\end{smallmatrix}\right)$ is the desired factorization.

\vskip5pt

{\bf Two out of three axiom} \ Let $f:A\longrightarrow B$ and $g:B\longrightarrow C$.

\vskip5pt

Suppose that $f, g\in \Weq$. Since $\Weq  =\TFib\circ \TCoFib$ and $\TCoFib$ is the class of splitting monomorphisms  with cokernel in ${\rm T}\mathcal C$, without loss of
generality, one may assume that there are  morphisms $(f, t_1): A\oplus U_1 \longrightarrow B$ and $(g, t_2):  B\oplus U_2 \longrightarrow C$ in $\TFib$,  with $U_1, U_2\in {\rm T}\mathcal C$. Then
\[
gf =(gf, g t_1 , t_2)\left(\begin{smallmatrix} 1\\0\\0\end{smallmatrix}\right)
\]
where $\left(\begin{smallmatrix} 1\\0\\0\end{smallmatrix}\right): A \longrightarrow A\oplus U_1\oplus U_2$ is in ${\rm TCoFib}$, and $(gf, g t_1 , t_2): A\oplus U_1\oplus U_2\longrightarrow C$. See the following diagram.

\[\xymatrix@R=0.5cm{
A\ar[dr]_-{\left(\begin{smallmatrix} 1\\ 0 \end{smallmatrix}\right)}\ar[rr]^-{f}& & B\ar[dr]_-{\left(\begin{smallmatrix}1\\ 0   \end{smallmatrix}\right)}\ar[rr]^-{g}&&C\\
&A\oplus U_1\ar[dr]_(.4){\left(\begin{smallmatrix}1&0\\0& 1\\0&0   \end{smallmatrix}\right)}\ar[ur]_-{\left(\begin{smallmatrix} f, t_1 \end{smallmatrix}\right)}&&B\oplus U_2\ar[ur]_-{\left(\begin{smallmatrix} g, t_2\end{smallmatrix}\right)}&\\
&&A\oplus U_1\oplus U_2\ar[ur]_(.65){\left(\begin{smallmatrix}f & t_1 & 0\\ 0 &0 &1   \end{smallmatrix}\right)} &&
}
\]
Since $\left(\begin{smallmatrix}f & t_1 & 0\\ 0 &0 &1   \end{smallmatrix}\right) = (f, t_1)\oplus {\rm Id}_{U_2}$ and $\TFib$ is closed under finite coproducts (cf. Lemma \ref{lemFWFS}(2)), it follows that $\left(\begin{smallmatrix}f & t_1 & 0\\ 0 &0 &1   \end{smallmatrix}\right)\in \TFib$.
Since $(gf, g t_1 , t_2)=(g,t_2)\left(\begin{smallmatrix}f & t_1 & 0\\ 0 &0 &1   \end{smallmatrix}\right)$
and $\TFib$ is closed under compositions (cf. Lemma \ref{lemFWFS}(2)), it follows that $(gf, g t_1 , t_2)=(g,t_2)\left(\begin{smallmatrix}f & t_1 & 0\\ 0 &0 &1   \end{smallmatrix}\right)\in \TFib$. Thus $gf  =(gf, g t_1 , t_2)\left(\begin{smallmatrix} 1\\0\\0\end{smallmatrix}\right)\in \TFib\circ \TCoFib = \Weq.$

\vskip5pt

Suppose that $f, gf\in \Weq$. Decompose $f=pi$  with $i:A\longrightarrow D$ in $\TCoFib$ and $p:D\longrightarrow B$ in $\TFib$. Let $t:U\longrightarrow C$ be a right ${\rm T}\mathcal C$-approximation. Then there is a commutative diagram
\[
\xymatrix{
& A\oplus U\ar[dl]_-{\left(\begin{smallmatrix} i & 0 \\0 & 1  \end{smallmatrix}\right)}\ar[d]^-{(gf,t)}\\
D\oplus U\ar[r]^-{(gp,t)}&C
}
\]
where $\left(\begin{smallmatrix} i & 0 \\0 & 1  \end{smallmatrix}\right)\in \TCoFib$, and $(gf,t)\in \TFib$ by Lemma \ref{Weq}. Thus $(gp,t)\in \TFib$ by Lemma \ref{TCoFibTFib}.
Consider $\left(\begin{smallmatrix} p & 0 \\0 & 1  \end{smallmatrix}\right): D\oplus U\longrightarrow B\oplus U$ and $(g, t): B\oplus U\longrightarrow C$.
Since $$(gp,t) = (g,t)\left(\begin{smallmatrix} p & 0 \\0 & 1  \end{smallmatrix}\right)= (g,t)(p\oplus {\rm Id}_U)$$
with $p\oplus {\rm Id}_U\in \TFib$ and $(gp,t)\in \TFib$, it follows from Definition \ref{FWFS}(2) that $(g,t)\in \TFib$. Since $g=(g,t)\left(\begin{smallmatrix} 1  \\0   \end{smallmatrix}\right)$, where
$\left(\begin{smallmatrix} 1  \\0   \end{smallmatrix}\right): B\longrightarrow B\oplus U$ is in $\TCoFib$, it follow that
$g\in \TFib \circ \TCoFib = \Weq$.

\vskip5pt

Finally, suppose that $g, gf\in \Weq$. Decompose $f=pi$ with $i:A\longrightarrow D$ in $\CoFib$ and $p:D\longrightarrow B$ in $\TFib$.
Then $p\in \TFib = \Fib\cap \Weq$, by Lemma \ref{intersection2}. Thus $gp\in \Weq$, as proven before. Let $t:U\longrightarrow C$ be a right ${\rm T}\mathcal C$-approximation. Then $\left(\begin{smallmatrix} i & 0 \\0 & 1  \end{smallmatrix}\right) = i\oplus {\rm Id}_U: A\oplus U\longrightarrow D\oplus U$ is in $\CoFib$, by Lemma \ref{lemFWFS}(2);  and  $(gp,t): D\oplus U\longrightarrow C\in \TFib$, by Lemma \ref{Weq};  and $(gf,t): A\oplus U\longrightarrow C$ is in $\TFib$, again by Lemma \ref{Weq}. Since
$$(gp, t)\left(\begin{smallmatrix} i & 0 \\0 & 1 \end{smallmatrix}\right) = (gf, t)\in \TFib$$
it follows from Lemma \ref{third2outof3} that $\left(\begin{smallmatrix} i & 0 \\0 & 1 \end{smallmatrix}\right)\in \TCoFib$, i.e., $\left(\begin{smallmatrix} i & 0 \\0 & 1 \end{smallmatrix}\right)$ is a splitting monomorphism, and hence $i$ is a splitting monomorphism with $\Coker i = \Coker \left(\begin{smallmatrix} i & 0 \\0 & 1 \end{smallmatrix}\right)\in {\rm T}\mathcal C$. Thus $i\in \TCoFib$. This shows that $f=pi\in \TFib \circ \TCoFib =\Weq$.

\vskip5pt

Up to now,  $(\CoFib, \Fib, \Weq)$ has proven to be a model structure. It is a fibrant model structure, by Theorem \ref{fibrant}. \hfill $\square$

\vskip5pt

\subsection {\bf Proof of Theorem \ref{correspondence}} \ Recall that $\Omega$ is the class of fibrantly weak factorization systems on $\A$,
and $\Gamma$ is the class of fibrant model structures on $\A$, and that
$$\Phi: \Omega \longrightarrow \Gamma \ \ \ \mbox{given by} \ \ \ (\CoFib, \TFib)\mapsto (\CoFib, \Fib, \Weq)$$
where $\Fib$ and $\Weq$ are given as in $(1.1)$,  and the map $$\Psi: \Gamma\longrightarrow \Omega  \ \ \ \mbox{given by} \ \ \
(\CoFib, \Fib, \Weq)\mapsto (\CoFib, \TFib)$$ Theorem \ref{mainthm} shows that ${\rm Im} \Phi\subseteq \Gamma$ and ${\rm Im} \Psi\subseteq \Omega$.
By Lemma \ref{intersection2} one has $\Psi\circ \Phi={\rm Id}_{\Omega}$. Let $(\CoFib, \Fib, \Weq)$ be a fibrant model structure.
By Theorem \ref{fibrant} one has
\begin{align*}
\Fib & = \{ \text{morphism} \ f \ | \ f \ \mbox{is} \ \Hom_{\mathcal A}({\rm T}\mathcal C, -)\mbox{-epic}\}\\
\TCoFib & =\{\text{splitting monomorphism} \ f \ |\ \Coker f\in {\rm T}\mathcal C \}.\end{align*}
Thus, by construction,  the image of $(\CoFib, \Fib, \Weq)$ under $\Phi\circ \Psi$ is just
$$(\CoFib, \Fib, \TFib\circ \TCoFib) = (\CoFib, \Fib, \Weq).$$
Thus $\Phi\circ \Psi={\rm Id}_{\Gamma}$. \hfill $\square$


\vskip5pt

\section{\bf Two applications}

Applying Theorem \ref{mainthm} to two special cases, we rediscover $\omega$-model structures and $\mathcal W$-model structures.
Both of them are fibrant, and not exact in general. We will analyze their relationships with exact model structures.

\subsection{$\omega$-model structures}

Let $\mathcal A$ be a weakly idempotent complete exact category,
$\mathcal{X}$ and $\mathcal{Y}$ full additive subcategories which are closed under direct summands and isomorphisms,
and $\omega =\mathcal{X}\cap \mathcal{Y}$.
Put
\begin{equation}\begin{aligned}\CoFib_\omega & = \{\text{inflation} \ f \ | \ \Coker f\in \mathcal X\}\\
\Fib_\omega & = \{\text{morphism} \ f \ | \ f \ \mbox{is} \ \Hom_{\mathcal A}(\omega, -)\mbox{-epic}\}\\
\TCoFib_\omega & = \{\text{splitting monomorphism} \ f \ | \ \Coker f\in \omega\}\\
\TFib_\omega & =\{\text{deflation} \ f \ | \ \Ker f\in \mathcal Y\}\\
\Weq_\omega & = \TFib_\omega \circ \TCoFib_\omega.\end{aligned} \end{equation}

\begin{thm}\label{omega-modelstructure} With the same notations as above, then $({\rm CoFib}_{\omega}, \ {\rm Fib}_{\omega}, \ {\rm Weq}_{\omega})$ given in $(5.1)$ is a  model structure with
$$\TCoFib_\omega  =  \CoFib_\omega\cap \Weq_\omega \ \ \mbox{and} \ \  \TFib_\omega  =  \Fib_\omega\cap \Weq_\omega$$ if and only if
$(\mathcal{X},\mathcal{Y})$ is a hereditary complete cotorsion pair in $\mathcal{A}$, and $\omega$ is contravariantly finite in $\mathcal A$.
If this is the case, then the class of cofibrant objects is $\mathcal X,$ the class of fibrant objects is $\mathcal A,$
and the class of trivial objects is $\mathcal Y;$ and ${\rm Ho}(\mathcal A)\cong \mathcal X/\omega$ as additive categories.
\end{thm}

\vskip5pt

Theorem \ref{omega-modelstructure} is initiated by A. Beligiannis and I. Reiten [BR, VIII, Theorem 4.2] on abelian categories.
It was called {\it an $\omega$-model structure}. The present form is given in [CLZ].
Theorem \ref{mainthm} gives a shorter proof of the ``if" part of Theorem \ref{omega-modelstructure}:
One can verifies that $(\CoFib_\omega, \TFib_\omega)$ is a fibrantly weak factorization system on $\mathcal{A}$,
and then $({\rm CoFib}_{\omega}, \ {\rm Fib}_{\omega}, \ {\rm Weq}_{\omega})$ is a fibrant model structure on $\mathcal{A}$, by Theorem \ref{mainthm};
also, ${\rm Ho}(\mathcal A)\cong \mathcal X/\omega$ as additive categories, by Theorem \ref{hoA}. For a complete proof we refer to \cite[Theorem 1.1]{CLZ}.

\vskip5pt

A fibrant model structure on an exact category is {\it weakly projective} if
cofibrations are exactly inflations with cofibrant cokernel, and each trivial fibration is a deflation; or equivalently,
trivial fibrations are exactly deflations with trivially fibrant kernel, and each cofibration is an inflation.
This is also equivalent to say that it is a fibrant model structure such that $(\mathcal{C}, {\rm T}\mathcal{F})$ is a complete cotorsion pair,
where $\mathcal C$ is the class of cofibrant objects, and $\rm T\mathcal{F}$ is the class of trivially fibrant objects.
See \cite[Proposition 5.2]{CLZ}.

\vskip5pt

As in abelian categories (\cite[VIII, Theorem 4.6]{BR}),
$\omega$-model structures are precisely weakly projective model structures,
on a weakly idempotent complete exact category, as shown by the following correspondence of Beligiannis and Reiten.

\begin{thm} \label{brcorrespondence} \ {\rm ([CLZ, Theorem 1.3])} \ Let $\A$ be a weakly idempotent complete exact category.
Then there is a one-to-one correspondence between hereditary complete cotorsion pairs with kernel $\omega$ contravariantly finite in $\A$,
and weakly projective model structures on $\A$, given by
$(\X,\Y)\mapsto (\CoFib_{\omega}, \Fib_{\omega}, \Weq_{\omega})$, with the inverse $(\CoFib, \Fib, \Weq)\mapsto (\mathcal{C}, \rm T\mathcal{F}).$
\end{thm}

\subsection{$\mathcal W$-model structure} For a class $\mathcal U$ of objects of category $\mathcal A$,
recall that a morphism $f$ of  $\mathcal A$ is said to be {\it $\Hom_{\mathcal{A}}(-, \mathcal{U})$-epic},
provided that $\Hom_{\mathcal{A}}(f, U)$ is epic for any $U\in \mathcal U$.

\vskip5pt

Let $\mathcal A$ be a weakly idempotent complete additive category, and $\mathcal W$ a full subcategory closed under isomorphisms. Put
\begin{equation}\begin{aligned}{\rm \CoFib}_{\mathcal W} & = \{\text{morphism} \ f \ | \ f \ \mbox{is} \ \Hom_{\mathcal A}(-, \mathcal W)\mbox{-epic}\}\\
\Fib_\mathcal W & = \{\text{morphism} \ f \ | \ f \ \mbox{is} \ \Hom_{\mathcal A}(\mathcal W, -)\mbox{-epic}\}\\
\TCoFib_\mathcal W & = \{\text{splitting monomorphism} \ f \ | \ \Coker f\in \mathcal W\}\\
\TFib_\mathcal W & =\{\text{splitting epimorphism} \ f \ | \ \Ker f\in \mathcal W\}\\
\Weq_\mathcal W & = \TFib_\mathcal W \circ \TCoFib_\mathcal W.\end{aligned} \end{equation}

\begin{thm}\label{Wms} {\rm([P, Proposition 2]; \cite[Theorem 4.5]{Bel})} \ Keep the notations in $(5.2)$. Then the followings are equivalent$:$

\vskip5pt

$(1)$ \ $({\rm CoFib}_{\mathcal{W}}, \ {\rm Fib}_{\mathcal{W}}, \ {\rm Weq}_{\mathcal{W}})$ is a model structure with
$$\TCoFib_\mathcal W = \CoFib_\mathcal W\cap \Weq_\mathcal W \ \ \mbox{and} \ \ \TFib_\mathcal W = \Fib_\mathcal W \cap \Weq_\mathcal W.$$

$(2)$ \ $\mathcal W$ is functorially finite in $\mathcal A$ and closed under finite coproducts and direct summands.

\vskip5pt

$(3)$ \ $({\rm CoFib}_{\mathcal W}, {\rm TFib}_{\mathcal W})$ is a fibrantly weak factorization system.

\vskip5pt

If this is the case, every object is cofibrant and fibrant, the class of trivial objects is $\mathcal W$, and ${\rm Ho}(\mathcal A)\cong \mathcal A/\mathcal W$ as additive categories.
\end{thm}

It is called {\it a $\mathcal W$-model structure},
where $\mathcal W$ refers to the class of trivial objects.
Theorem \ref{mainthm} gives a shorter proof of Theorem \ref{Wms}.
\vskip5pt

A model structure $(\CoFib, \Fib, \Weq)$ on pointed category $\mathcal A$
is {\it bifibrant} if $\mathcal C = \mathcal A = \mathcal F$. As shown by the following correspondence,
$\mathcal W$-model structures are precisely bifibrant model structures,
on a weakly idempotent complete additive category.

\begin{cor} \label{Wcorrespondence} {\rm ([P, Corollary 5], [Bel, Theorem 4.6])} \ Let $\A$ be a weakly idempotent complete additive category.
Then there is a one-to-one correspondence between functorially finite full subcategories closed under finite coproducts, direct summands and isomorphisms,
and bifibrant model structures on $\A$, given by
$\mathcal{W}\mapsto (\CoFib_{\mathcal{W}}, \Fib_{\mathcal{W}}, \Weq_{\mathcal{W}})$, with the inverse $(\CoFib, \Fib, \Weq)\mapsto \mathcal W$.
\end{cor}

After we got  Theorems \ref{Wms} and Corollary \ref{Wcorrespondence}, we found that they  have been  discovered by T. Pirashvili \cite{P} for abelian categories,
and by A. Beligiannis in \cite{Bel} for idempotent complete additive categories.
Recall that an additive category is {\it idempotent complete}, if each idempotent morphism has a kernel.
Clearly, an idempotent complete additive category is weakly idempotent complete, but the converse is not true. See [B\"u, Exercise 7.11].

\subsection{The relationships} \ Let $\A$ be an exact category.
An exact model structure is {\it projective} (respectively, {\it injective}), if each object is fibrant (respectively, cofibrant).
An exact model structure on $\mathcal A$ is projective (respectively, injective) if and only if
$\mathcal A$ has enough projective (respectively, injective) objects and
trivially cofibrant (respectively, trivially fibrant) objects coincide with projective (respectively, injective) objects. See [G2, 4.6].

\vskip5pt

An exact category is {\it Frobenius}, if it has enough projective objects and enough injective objects, and
the projective objects coincide with the injective objects.

\vskip5pt

An exact model structure on $\mathcal A$ is  {\it Frobenius}, if it is both projective and injective.
Thus, an exact model structure on $\mathcal A$ is Frobenius if and only if
$\mathcal A$ is a Frobenius category and  trivial objects coincide with the projective-injective objects of $\mathcal A$. See [G2, 4.7].

\begin{prop}\label{threeintersections} \ Let $\A$ be a weakly idempotent complete exact category.

\vskip5pt

$(1)$ \   An $\omega$-model structure on $\mathcal A$ is exact
if and only if $\mathcal A$ has enough projective objects and $\omega = \mathcal P$, the class of projective objects of $\mathcal A;$
and if and only if this $\omega$-model structure is projective.

\vskip5pt

If this is the case, then
\begin{align*}& \CoFib_\omega = \{\text{inflation} \ f \ | \ \Coker f\in \mathcal X\}, \ \ \ \ \ \   \Fib_\omega  = \{\text{deflation}\}\\
&\TCoFib_\omega  = \{\text{splitting monomorphism} \ f \ | \ \Coker f\in \mathcal P\}\\
&\TFib_\omega  =\{\text{deflation} \ f \ | \ \Ker f\in \mathcal Y\}, \ \ \ \ \ \ \ \ \ \ \ \Weq_\omega = \TFib_\omega \circ \TCoFib_\omega.\end{align*}

$(2)$ \ A $\mathcal W$-model structure on $\mathcal A$ is an $\omega$-model structure if and only if
$\mathcal A$ has enough injective objects and $\mathcal W =\mathcal I$, the class of injective objects of $\mathcal A$.

\vskip5pt

If this is the case, then $\omega = \mathcal W = \mathcal I$ is contravariantly finite in $\mathcal A$ and
\begin{align*}&\CoFib_\mathcal W  = \{\text{inflation}\}, \ \ \ \ \ \ \ \ \ \ \ \ \ \ \ \ \ \ \Fib_\mathcal W  = \{f \ | \ f \ \mbox{is} \ \Hom_{\mathcal A}(\mathcal I, -)\mbox{-epic}\}\\
&\TCoFib_\mathcal W  = \{\text{splitting monomorphism} \ f \ | \ \Coker f\in \mathcal I\}\\
&\TFib_\mathcal W  =\{\text{splitting epimorphism} \ f \ | \ \Ker f\in \mathcal I\}, \ \ \ \ \  \Weq_\mathcal W = \TFib_\omega \circ \TCoFib_\omega.\end{align*}

\vskip5pt

$(3)$ \  Let \ $(\CoFib, \Fib, \Weq)$  be a model structure on $\mathcal{A}$. Then the following are equivalent$:$

\vskip5pt

\hskip15pt ${\rm (i)}$  \ $(\CoFib, \Fib, \Weq)$  is both an exact model structure and a $\mathcal W$-model structure$;$

\vskip5pt

\hskip15pt ${\rm (ii)}$  \ $(\CoFib, \Fib, \Weq)$ is a Frobenius model structure$;$

\vskip5pt

\hskip15pt ${\rm (iii)}$    \ $(\CoFib, \Fib, \Weq)$  is simultaneously an exact model structure, an $\omega$-model structure, and a $\mathcal W$-model structure.

\vskip5pt

If this the case, then $\omega = \mathcal W = \mathcal P = \mathcal I$, and
$$\CoFib = \{\mbox{inflation}\},  \ \ \  \Fib = \{\mbox{deflation}\},  \ \ \ {\rm Weq} = \{ \ f \ | \ \bar f \ \mbox{is an isomorphism in} \ \mathcal{A}/ \mathcal P\}.$$
\end{prop}

\begin{proof}  $(1)$ \ If $\omega$-model structure $(\CoFib_{\omega}, \Fib_{\omega}, \Weq_{\omega})$ is exact,
then by the Hovey correspondence $(\mathcal X, \ \A, \ \Y)$ is a Hovey triple,
thus $(\omega, \mathcal{A})$ is a complete cotorsion pair, and hence $\mathcal A$ has enough projective objects and $\omega = \mathcal P$.

\vskip5pt

Conversely, \ assume that $\mathcal A$ has enough projective objects and $\omega = \mathcal P$. For any $f\in \Fib_\omega$, by definition $f: A\longrightarrow B$ is $\Hom_{\mathcal A}(\omega, -)$-epic.
Taking a deflation $g: P\longrightarrow B$ with $P\in\mathcal P$, then
$g = fh$ for some $h: P\longrightarrow A$. Since $\mathcal{A}$ is weakly idempotent complete, $f$ is a deflation. Thus
$${\rm Fib}_{\omega} = \{ f \ | \ f \ \mbox{is} \ \Hom_{\mathcal A}(\omega, -)\mbox{-epic}\} = \{\text{deflation}\}.$$
Together with $\CoFib_\omega = \{\text{inflation} \ f \ | \ \Coker f\in \mathcal X\}$,  by definition $(\CoFib_{\omega}, \Fib_{\omega}, \Weq_{\omega})$ is exact.

\vskip5pt

$(2)$ \ If a $\mathcal W$-model structure is an $\omega$-model structure,
then $(\mathcal C, {\rm T}\mathcal F) = (\mathcal A, \mathcal W)$ is a complete cotorsion pair, it follows that
$\mathcal A$ has enough injective objects and $\mathcal W = \mathcal I$.

\vskip5pt

Conversely, assume that $\mathcal A$ has enough injective objects and $\mathcal W = \mathcal I$.
For any $f\in {\rm \CoFib}_{\mathcal W}$, by definition  $f: A\longrightarrow B$ is $\Hom_{\mathcal A}(-, \mathcal I)\mbox{-epic}$.
Taking an inflation $g: A\longrightarrow I$ with $I\in\mathcal I$, then
$g = hf$ for some $h: B\longrightarrow I$. Since $\mathcal{A}$ is weakly idempotent complete, $f$ is an inflation. Thus
$${\rm CoFib}_{\mathcal W} = \{ f \ | \ f \ \mbox{is} \ \Hom_{\mathcal A}(-, \mathcal I)\mbox{-epic}\} = \{\text{inflation}\}.$$
Since $\TCoFib_\mathcal W = \{\text{splitting monomorphism} \ f \ | \ \Coker f\in \mathcal I\}$, each trivial fibration is a deflation.
By definition, $(\CoFib_{\mathcal{W}}, \Fib_{\mathcal{W}}, \Weq_{\mathcal{W}})$ is a weakly projective model structure, and hence an $\omega$-model structure.

\vskip5pt

$(3)$ \ ${\rm (i)} \Longrightarrow {\rm (ii)}:$  \ Since $\mathcal W$-model structures are exactly bifibrant model structures, this implication follows from the definition
of a Frobenius model structure.

\vskip5pt

\hskip15pt ${\rm (ii)}\Longrightarrow {\rm (iii)}:$  \ If $(\CoFib, \Fib, \Weq)$ is a Frobenius model structure,
then it is a $\mathcal W$-model structure, $\mathcal A$ has enough injective objects and $\mathcal W =\mathcal I$.
By $(2)$, it is an $\omega$-model structure with $\omega = \mathcal W =\mathcal I.$

\vskip5pt

\hskip15pt ${\rm (iii)}\Longrightarrow {\rm (i)}:$   This is trivial.
\end{proof}

A model structure on pointed category $\mathcal A$ is {\it non-trivial}, if $\Ho(\mathcal A)\ne 0$.
We only interested in non-trivial model structures, and all examples of model structures below are non-trivial.

\begin{exms}\label{examp}

$(1)$ \ There are many examples of exact model structures which are also $\omega$-model structure. For example,
let $\mathcal A$ be an abelian category with enough projective objects, and ${\rm Ch}^-(\mathcal A)$ the category of right bounded complexes of objects of $\A$.
Then $(\CoFib, \Fib, \Weq)$ is an abelian model structure on ${\rm Ch}^-(\mathcal A)$,
where $$\CoFib = \{\mbox{monomorphism} \ f \ | \ \Coker f \ \mbox{is a complex of projective objects}\}$$
$$\Fib = \{\mbox{epimorphism} \}, \ \ \ \ \ \ \Weq = \{\mbox{quasi-isomorphism}\}.$$
See [Q1, page 1.2]. The class of cofibrant (fibrant, trivial, respectively) objects is given by
$$\mathcal C = \{\mbox{complex of projective objects}\}, \ \ \ \mathcal F = {\rm Ch}^-(\mathcal A), \ \ \  \mathcal W = \{\mbox{acyclic complex}\}.$$
By definition, it is a weakly projective model structure, and hence an $\omega$-model structure, by Theorem \ref{brcorrespondence}. But it is not a $\mathcal W$-model structure.

\vskip5pt

In fact, $(\mathcal C, \mathcal W)$ is a hereditary and complete cotorsion pair in ${\rm Ch}^-(\mathcal A)$, and
$\omega: = \mathcal C \cap \mathcal W$ is exactly the class of projective objects of ${\rm Ch}^-(\mathcal A)$. See [EJX] and [G1].

\vskip5pt

$(2)$ \ Let $A$ be an artin algebra, $A\mbox{-mod}$ the category of finitely generated $A$-modules, and
$\mathcal I$ the full subcategory of injective modules.
Then $\mathcal I$ is functorially finite in $A\mbox{-mod}$, and $(\CoFib, \Fib, \Weq)$ is both an $\omega$-model structure and a $\mathcal W$-model structure on $A\mbox{-mod}$,
where
\begin{align*}& {\rm \CoFib}  = \{\text{monomorphism}\}, \ \ \ \ \ \ \ \ \ \ \Fib  = \{\text{morphism} \ f \ | \ f \ \mbox{is} \ \Hom_{\mathcal A}(\mathcal I, -)\mbox{-epic}\}\\
& \TCoFib  = \{\text{splitting monomorphism} \ f \ | \ \Coker f\in \mathcal I\}\\
& \TFib =\{\text{splitting epimorphism} \ f \ | \ \Ker f\in \mathcal I\}, \ \ \ \ \ \ \Weq  = \TFib \circ \TCoFib.\end{align*}
Cofibrant objects, fibrant objects, and trivial objects are respectively given by
$$\mathcal C=\mathcal F=A\mbox{-mod}, \ \ \ \ \mathcal W=\mathcal I.$$

\vskip5pt

If $A$ is not self-injective, then $(\CoFib, \Fib, \Weq)$ is not an abelian model structure since $(\mathcal I, A\mbox{-mod})$ is not a cotorsion pair.

\vskip5pt

If $A$ is self-injective, then $(\CoFib, \Fib, \Weq)$ is a Frobenius model structure, i.e.,
it is simultaneously an $\omega$-model structure, a $\mathcal W$-model structure, and an abelian model structure.

\vskip5pt

$(3)$ \ Let $R$ be a ring, and ${\rm Ch}(R)$ the complex category of $R$-modules.
A complex $I$ is {\it {\rm dg}-injective}, if each $I^n$ is an injective $R$-module and ${\rm Hom}^\bullet(E, I)$ is acyclic for any acyclic complex $E$ ([AF]).
Denote by ${\rm dg}\mathcal I$ the class of {\rm dg}-injective complexes.
Then $(\CoFib, \Fib, \Weq)$ is an abelian model structure on ${\rm Ch}(R)$,
where $$\CoFib = \{\mbox{monomorphism}\}, \ \ \Fib = \{\mbox{epimorphism} \ f \ | \ \Ker f\in {\rm dg}\mathcal I\}, \ \ \Weq = \{\mbox{quasi-isomorphism}\}.$$
See \cite[Theorem 2.3.13]{H1}. Cofibrant objects, fibrant objects, and trivial objects, are respectively given by
$$\mathcal C = {\rm Ch}(R), \ \ \ \mathcal F = {\rm dg}\mathcal I, \ \ \ \mathcal W = \{\mbox{acyclic complex} \}.$$
Thus, it is not a fibrant model structure, and hence it is neither an $\omega$-model structure nor a $\mathcal W$-model structure.

\vskip5pt

$(4)$ \ Let $A$ be an artin algebra.
For a module $M$, let ${\rm add} M$ be the set of modules which are summands of finite direct sums of copies of $M$,
and $\widetilde{{\rm add} M}$ the set of modules $X$ such that there is an exact sequence
$$0\longrightarrow X\longrightarrow M^0\longrightarrow \ldots\longrightarrow M^s\longrightarrow 0$$
with each $M^i\in {\rm add} M$.

\vskip5pt

Let $T$ be a tilting module, i.e., ${\rm proj.dim.} T <\infty$, \ $\Ext_A^i(T, T)=0$ for $i\ge 1$, and  $A\in \widetilde{{\rm add}T}$.
By \cite[VIII, 4.13]{BR},  $(\widetilde{{\rm add} T}, \ T^{\perp_{\ge 1}})$ is a hereditary complete cotorsion pair in $A\mbox{-mod}$ with $\omega:=\widetilde{{\rm add} T}\cap T^{\perp_{\ge 1}}={\rm add} T$ contravariantly finite in $A\mbox{-mod}$, where $T^{\perp_{\ge 1}}=\{M \in A\mbox{-mod} \ | \ \Ext^i_A(T, M) = 0, \ \forall \ i\ge 1\}$.
Then $(\CoFib, \Fib, \Weq)$ is an $\omega$-model structure on $A\mbox{-mod}$, where
\begin{align*}& {\rm \CoFib}  = \{\text{inflation} \ f \ | \ \Coker f\in \widetilde{{\rm add} T}\}, \ \ \ \ \ \ \ \Fib  = \{f \ | \ f \ \mbox{is} \ \Hom_{\mathcal A}({\rm add} T, -)\mbox{-epic}\}\\
& \TCoFib  = \{\text{splitting monomorphism} \ f \ | \ \Coker f\in {\rm add} T\}\\
& \TFib =\{\text{deflation} \ f \ | \ \Ker f\in T^{\perp_{\ge 1}}\}, \ \ \ \ \ \ \ \ \ \  \Weq  = \TFib \circ \TCoFib.\end{align*}
Cofibrant objects, fibrant objects, and trivial objects are respectively given by
$$\mathcal C=\widetilde{{\rm add} T},\ \ \ \ \mathcal F=A\mbox{-mod}, \ \ \ \ \mathcal W=T^{\perp_{\ge 1}}.$$

Now, assume that $T$ is not projective. Since $({\rm add} T, A\mbox{-mod})$ is not a cotorsion pair in $A\mbox{-mod}$,
$(\CoFib, \Fib, \Weq)$ is not an exact model structure.

\vskip5pt

Assume in addition that $A$ is of infinite global dimension. Then $\widetilde{{\rm add} T}\ne A\mbox{-mod}$, and hence $(\CoFib, \Fib, \Weq)$ is not a $\mathcal W$-model structure.

\vskip5pt

$(5)$ \ Let $A$ be an artin algebra, and $M$ an $A$-module.
Then ${\rm add} M$ is functorially finite in $A\mbox{-mod}$.
By Theorem \ref{Wms}, $(\CoFib, \Fib, \Weq)$ is a $\mathcal W$-model structure on $A\mbox{-mod}$,
where
\begin{align*}& \CoFib  = \{f \ | \ f \ \mbox{is} \ \Hom_{\mathcal A}(-,{\rm add} M)\mbox{-epic}\}, \ \ \ \ \ \ \ \ {\rm \Fib}  = \{ f \ | \ f \ \mbox{is} \ \Hom_{\mathcal A}({\rm add} M, -)\mbox{-epic}\}\\
& \TCoFib  = \{\text{splitting monomorphism} \ f \ | \ \Coker f\in {\rm add} M\}\\
& \TFib =\{\text{splitting epimorphism} \ f \ | \ \Ker f\in {\rm add} M\}, \ \ \ \ \ \ \Weq  = \TFib \circ \TCoFib.\end{align*}
Cofibrant objects, fibrant objects, and trivial objects are respectively given by
$$\mathcal C=\mathcal F=A\mbox{-mod}, \ \ \ \ \mathcal W={\rm add} M.$$

Now, choose $M$ to be a non injective module.
Since $(A\mbox{-mod}, {\rm add} M)$ is not a cotorsion pair in $A\mbox{-mod}$,  it follows that $(\CoFib, \Fib, \Weq)$ is neither an abelian model structure nor an $\omega$-model structure on $A\mbox{-mod}$.
\end{exms}

\section{\bf The dual version: cofibrant model structures}
We state the following dual version without proofs. A model structure on a pointed category is called {\it cofibrant},
if every object is a cofibrant object.

\begin{thm}\label{cofibrant} \ Let  $(\CoFib, \Fib, \Weq)$ be a model structure on weakly idempotent complete additive category $\mathcal A$.
Then the followings are equivalent.

  \vskip5pt

  $(1)$ \ $(\CoFib, \Fib, \Weq)$ is a cofibrant model structure.

  \vskip5pt

  $(2)$ \ $\TFib   =\{\text{splitting epimorphism} \ f \ |\ \Ker f\in {\rm T}\mathcal F \}$.

\vskip5pt

  $(3)$ \ $\CoFib  = \{ \text{morphism} \ f \ | \ f \ \mbox{is} \ \Hom_{\mathcal A}(-, {\rm T}\mathcal F)\mbox{-epic}\}$.

\vskip5pt

If this is the case, then $\Ho(\mathcal A)\cong \mathcal F/{\rm T}\mathcal F$, as additive categories. \end{thm}

To construct cofibrant model structure, we need

\begin{defn}\label{CWFS} \ Let $\mathcal A$ be a pointed category.
A pair $({\rm L}, {\rm R})$ of classes of morphisms is a cofibrantly weak factorization system on $\mathcal A$, provided that the following conditions are satisfied:

\vskip5pt

$(1)$ \ $({\rm L}, {\rm R})$ is a weak factorization system.

\vskip5pt

$(2)$ \ For $f: A\longrightarrow B$ and $g:B\longrightarrow C$, if $g, gf\in {\rm L}$, then $f\in {\rm L}$.

\vskip5pt

$(3)$ \ The class $\mathcal L\cap \mathcal R$ of objects is covariantly finite in $\mathcal A$,
where $$\mathcal L =\{X \ | \ 0\longrightarrow X \ \mbox{is in} \ {\rm L}\}, \ \ \ \ \mathcal R =\{X \ | \ X\longrightarrow 0 \ \mbox{is in} \ {\rm R}\}.$$
\end{defn}

\begin{prop}\label{cofibrantFWFS} \ Let $(\CoFib, \Fib, \Weq)$ be a cofibrant model structure. Then both $(\TCoFib, \Fib)$ and $(\CoFib, \TFib)$ are cofibrantly weak factorization systems.
\end{prop}

\vskip5pt

Assume that $(\TCoFib, \Fib)$ is a cofibrantly weak factorization system on $\mathcal A$.  Put
\begin{equation}\begin{aligned}
\mathcal F & = \{ \mbox{object} \ X \ | \ \mbox{morphism} \ X\longrightarrow 0 \ \mbox{is in} \ \Fib\} \\
\mathcal W & = \{ \mbox{object} \ X \ | \ \mbox{morphism}  \ 0\longrightarrow X \ \mbox{is in} \ \TCoFib\} \\
{\rm T}\mathcal F & = \mathcal F\cap \mathcal W\\
\CoFib & = \{ \text{morphism} \ f \ | \ f \ \mbox{is} \ \Hom_{\mathcal A}(-, {\rm T}\mathcal F)\mbox{-epic}\}
\\
\TFib  & =\{\text{splitting epimorphism} \ f \ |\ \Ker f\in {\rm T}\mathcal F \}\\
\Weq & = \TFib\circ\TCoFib.\end{aligned}\end{equation}

\begin{thm}\label{dmainthm} \ Let $\mathcal{A}$ be a weakly idempotent complete additive category.
If  $(\CoFib, \Fib, \Weq)$ is a cofibrant model structure on $\mathcal A$, then
$(\TCoFib, \Fib)$ is a cofibrantly weak factorization system on $\mathcal A$, and
$\CoFib$ and $\TFib$ are given as in $(6.1)$.

\vskip5pt

Conversely, if $(\TCoFib, \Fib)$ is a cofibrantly weak factorization system on $\mathcal A$, then $(\CoFib, \Fib,$ $\Weq)$ is a cofibrant model structure on $\mathcal A$,
where $\CoFib$ and $\Weq$ are given as in $(6.1)$.
\end{thm}

\begin{thm} \label{dcorrespondence} \ Let $\A$ be a weakly idempotent complete additive category, $\Omega$ the class of cofibrantly weak factorization systems on $\A$,
and $\Gamma$ the class of cofibrant model structures on $\A$. Then the map
$$\Phi: \Omega \longrightarrow \Gamma \ \ \ \mbox{given by} \ \ \ (\TCoFib, \Fib)\mapsto (\CoFib, \Fib, \Weq)$$
where $\CoFib$ and $\Weq$ are given as in $(6.1)$,  and the map $$\Psi: \Gamma\longrightarrow \Omega  \ \ \ \mbox{given by} \ \ \
(\CoFib, \Fib, \Weq)\mapsto (\TCoFib, \Fib)$$
give a one-one correspondence between $\Omega$ and $\Gamma$.
\end{thm}

\end{document}